\documentclass[12pt
]{article}
\usepackage{amsmath,amssymb,amsfonts,amsthm,url,xspace,graphicx,color,mathrsfs,enumitem}
\usepackage[normalem]{ulem}
\usepackage{tikz}
\usepackage{bm}
\usepackage[colorlinks=true,linkcolor=blue,citecolor=blue,urlcolor=blue,pdfborder={0 0 0}]{hyperref}
\usepackage{graphicx}

\usepackage{xcolor} 

\theoremstyle{plain}
\newtheorem{thm}{Theorem}[section]
\newtheorem{prop}[thm]{Proposition}

\newtheorem{cor}[thm]{Corollary}

\newcommand{\eps}{\varepsilon}
\newcommand{\C}{\mathbb{C}}
\newcommand{\Z}{\mathbb{Z}}

\newcommand{\G}{\mathcal{G}}

\newcommand{\K}{\mathbb{K}}

\newcommand{\R}{\mathbb{R}}

\newcommand{\M}{\mathrm{M}}
\newcommand{\Mc}{\mathcal{M}}
\newcommand{\Lc}{\mathcal{L}}

\newcommand{\Tr}{\mathbf{Tr}}
\newcommand{\tr}{\mathrm{Tr}}

\newcommand{\SL}{\mathrm{SL}_n}
\newcommand{\GL}{\mathrm{GL}_n}
\newcommand{\SLth}{\mathrm{SL}_3}
\newcommand{\CC}{\mathcal{C}}

\newcommand{\be}{\begin{equation}}
\newcommand{\ee}{\end{equation}}

\newcommand{\old}[1]{}

\usepackage[
	,marginparwidth=75pt,
]{geometry}
	\usepackage{pdfpages} 
	\usepackage{comment}  
		\includecomment{note}	
\usepackage{setspace} 

\theoremstyle{definition}

\title{Dimers, webs, and
local systems}
\author{Daniel C. Douglas, Richard Kenyon, Haolin Shi}

\begin{document}

\maketitle

\abstract{
For a planar bipartite graph $\G$ equipped with a $\SL$-local system, we show that the 
determinant of the associated Kasteleyn matrix counts ``$n$-multiwebs" (generalizations of $n$-webs) in $\G$, weighted by their web-traces. We use this fact to study 
random $n$-multiwebs in graphs on some simple surfaces.}

\section{Introduction}

\subsection{Multiwebs}
\label{ssec:multiwebs}

Let $\G=(V=B\cup W,E)$ be a bipartite graph.
An \emph{$n$-multiweb $m$ of $\G$} is a multiset of edges with degree $n$ at each vertex,
that is, a mapping $m:E\to\{0,1,2,\dots\}$ such that for each vertex $v\in\G$ we have $\sum_{u\sim v} m_{uv}=n$.
In other words each vertex is an endpoint of exactly $n$ edges of $m$, counted with multiplicity. 
See Figure \ref{3multiweb} for a $3$-multiweb in a grid.  We let $\Omega_n(\G)$ be the set of $n$-multiwebs of $\G$. 
The existence of an $n$-multiweb forces $\G$ to be \emph{balanced},
that is, $\G$ has the same number $N$ of white vertices as black vertices: $V=B\cup W$ with $N=|B|=|W|$.

The notion of $n$-multiweb was introduced in \cite{Fraser19TransAmerMathSoc}, where the terminology \emph{$n$-weblike subgraph} was used.
Multiwebs are generalizations of what, in the literature, are called \textit{webs}, namely regular bipartite graphs.
\begin{figure}
\center{\includegraphics[width=2in]{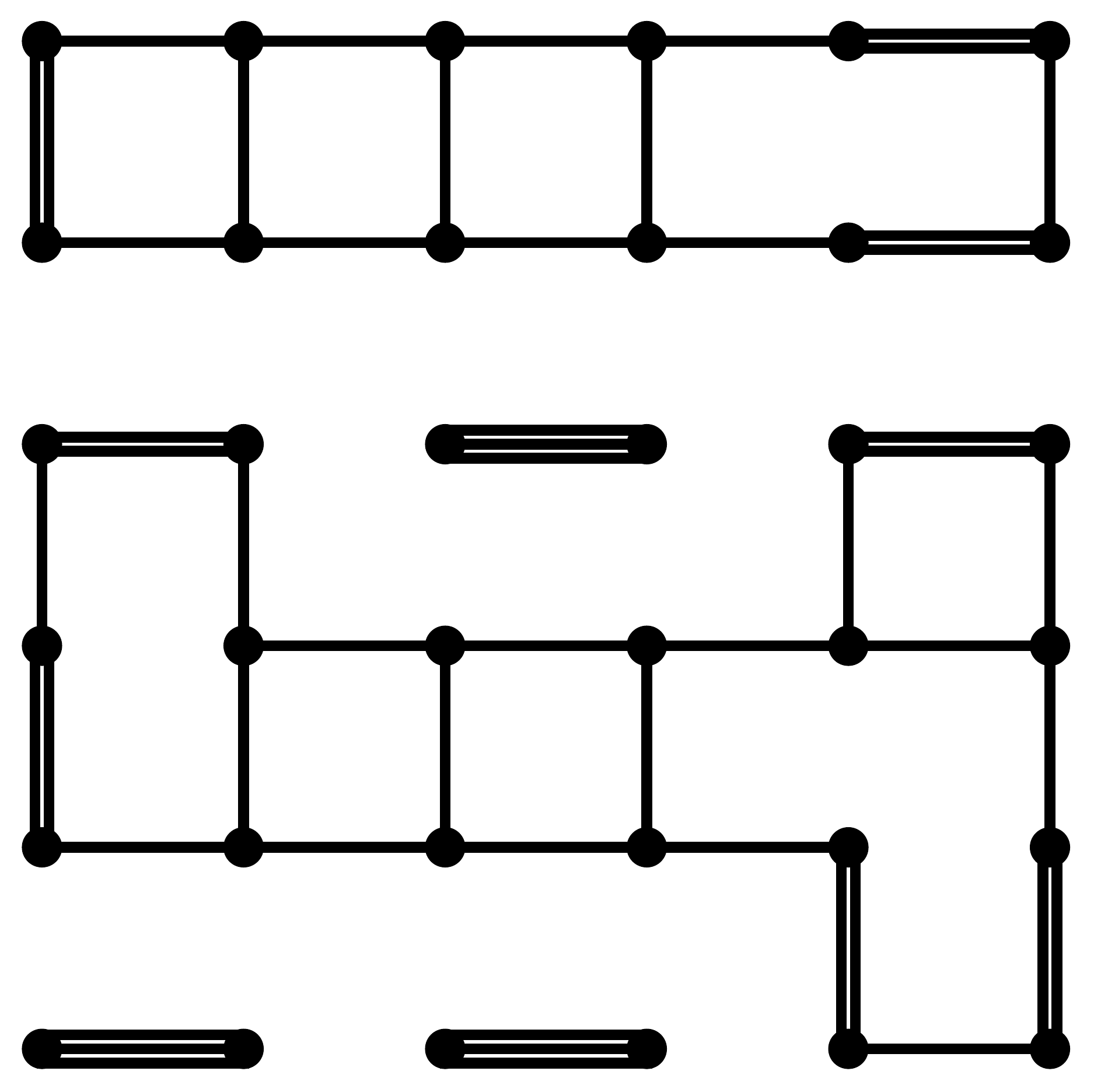}}
\caption{\label{3multiweb}A $3$-multiweb in the $6\times6$ grid.}
\end{figure}

When $n=1$, an $n$-multiweb of $\G$ is a \emph{dimer cover} of $\G$, also known as a \emph{perfect matching} of $\G$.
When $n=2$ an $n$-multiweb of $\G$ is a \emph{double dimer cover}. Dimer covers and double dimer covers are classical combinatorial objects studied starting in the 
1960's by Kasteleyn \cite{Kasteleyn61Physica}, Temperley/Fisher \cite{Temperley61PhilosMag}, and many others, see e.g. \cite{Kenyon09Statisticalmechanics} for a survey. Our goal here is to
study $n$-fold dimer covers, or equivalently, $n$-multiwebs, for $n\ge 3$.

In \cite{Kenyon14CMP}, $\text{SL}_2$-local systems
were used to study topological properties of double dimer covers on planar graphs. We extend this here to $\SL$-local systems for $n\ge 3$.
On a bipartite graph $\G$ on a surface with an $\SL$-local system $\Phi$, we define the \emph{trace} $\Tr(m)=\Tr(m,\Phi)$ of an $n$-multiweb $m$, a small generalization
of the trace of an $n$-web common in the literature. 
Traces of webs are used in the study of tensor networks,
representation theory, cluster algebras, and knot theory \cite{Jaeger92DiscreteMath, Kuperberg96CommMathPhys, SikoraTrans01, Morse10Involve, FominAdvMath16, Fraser19TransAmerMathSoc}. In this paper we
study these traces from a probabilistic and combinatorial point of view. 

To distinguish our trace from the trace of a matrix we should in principle
refer to it as a \emph{web-trace}. However we say ``trace" when there is no risk of confusion (we need to be careful precisely when the multiweb is a loop, because the web-trace for a $\GL$ connection
is not generally equal to the trace of the associated monodromy around the loop; see Section \ref{loop}).  

Our main result computes the determinant of a certain operator $K(\Phi)$, the \emph{Kasteleyn matrix} for the planar bipartite graph $\G$ in the presence of an 
$\SL$-local system $\Phi$, as a sum of traces of multiwebs:
\bigskip

\noindent{\bf Theorem. } Up to a global sign,
$ \det \tilde K(\Phi) = \sum_{n\text{-multiwebs } m \in \Omega_n(\G)} \Tr(m)$.
\bigskip

Here $\tilde K \in \M_{nN}(\R)$ (the set of $nN$ by $nN$ real matrices) is obtained from $K \in \M_N(\M_n(\R))$ in the obvious way.
For the definitions and precise statement, see below and Theorem \ref{main}.
This theorem holds more generally for $\mathrm{M}_n$-connections (connections with parallel transports in $\mathrm{M}_n(\R) = \text{End}(\R^n)$). 

In the case $n=1$, we have $\Tr(m) = 1$  for any $1$-multiweb $m$ for an $\text{SL}_1$-local system, or simply the product of edge weights for an $\mathrm{M}_1$-connection; in this case $K$ is the usual Kasteleyn matrix.  In this sense 
our result generalizes Kasteleyn's theorem from \cite{Kasteleyn61Physica}. 

In the case $n=2$, we give a new proof of (a slightly more general version of) a theorem of \cite{Kenyon14CMP} enumerating double-dimer covers, see Section \ref{ddsection}.

As another application of the theorem, in the case $n=3$ we show how to enumerate isotopy classes of ``reduced" $3$-multiwebs (see below), 
on either the annulus or the pair of pants (see Sections \ref{annulussection} and \ref{pantssection}).

\subsection{Colorings}
\label{ssec:colorings}
For the \emph{identity} connection $\Phi\equiv I$ (where the identity matrix is assigned to every edge), the trace of an $n$-multiweb has a simple combinatorial interpretation. The trace
for the identity connection is a signed count of the number of edge-$n$-colorings (see Proposition \ref{tracecolor} below),
and in fact for planar multiwebs, it is the unsigned number of edge-$n$-colorings (see Proposition \ref{tracecolorplanar} below).
Here, an \emph{edge-$n$-coloring} of an $n$-multiweb $m$ is a coloring of the edges of 
$m$ with colors from $\CC=\{1,2,\dots,n\}$ so that at each vertex all $n$ colors are present.
More precisely, an edge-$n$-coloring is a map from the edges of $m$ into $2^{\CC}$, the set of subsets of $\CC$, 
with the property that, first, 
an edge of multiplicity $k$ maps to a subset of $\CC$ of size $k$, and secondly, the union of the color sets over all edges at
a vertex is $\CC$. For example, the multiweb appearing in Figure \ref{3multiweb} has five components. We can calculate that there are 24 ways to color the component on the top and 48 ways to color the other nontrivial component. There is a unique way to color a tripled edge. So $m$ has $48*24*1*1*1=1152$ edge-$3$-colorings and the trace of this multiweb is 
$\Tr(m,I)=1152$ .

\begin{figure}
\center{\includegraphics[width=2in]{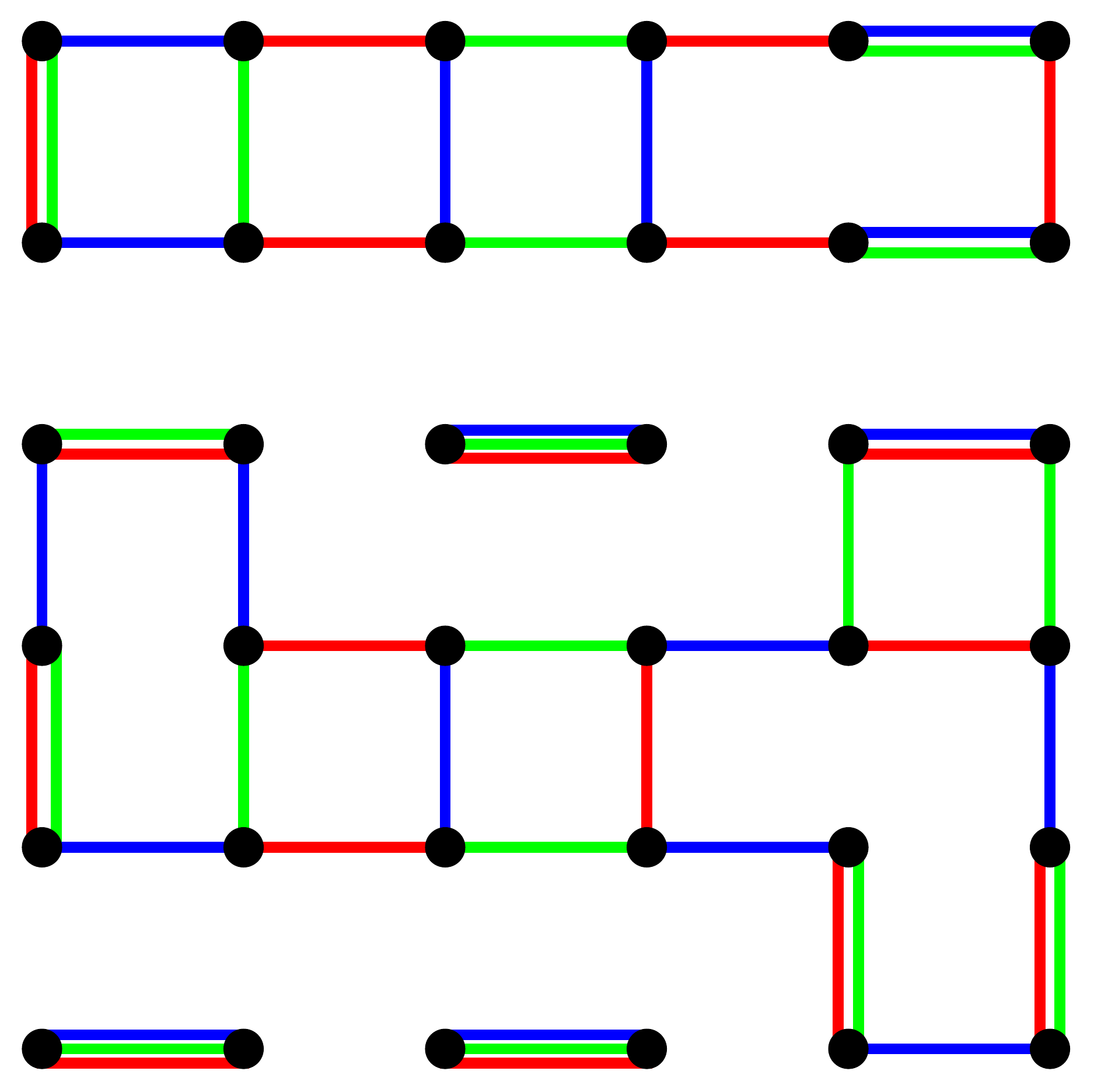}}
\caption{\label{colored3mw}An edge-$3$-coloring of the multiweb of Figure \protect{\ref{3multiweb}}.}
\end{figure}

For a planar graph $\G$, we define the \emph{partition function of $n$-multiwebs} $Z_{nd}$ (here $nd$ stands for ``$n$-dimer"),  to be 
\be\label{Zndtrivial}Z_{nd} := \sum_{m\in\Omega_n(\G)} \Tr_+(m, I).\ee
The $+$ subscript corresponds to a choice of \emph{positive cilia}, see Section \ref{planarwebscn} below and comments after 
Corollary \ref{poscilia}.
We also define $Z_d$ to be the number of single dimer covers of $\G$.

\begin{prop}\label{thm1}
$Z_{nd} = (Z_d)^n$.
\end{prop}

\begin{proof}
There is a natural map from ordered $n$-tuples of single dimer covers to $n$-multiwebs, obtained by taking the union and recording multiplicity over each edge. The fiber over a fixed $n$-fold dimer cover
is the number of its edge-$n$-colorings, which is $\Tr_+(m, I)$. 
\end{proof}

Associated to $Z_{nd}$ is a natural probability measure $\mu_n$ on $n$-multiwebs of $\G$, where a multiweb $m$ has probability
$\Pr(m) = \frac{\Tr_+(m, I)}{Z_{nd}}.$ 
One of our motivations is to analyze this measure; as a tool to probe this measure we use $\SL$-local systems and our main theorem, Theorem \ref{main}.

\subsection{The case \texorpdfstring{$n=3$}{n=3}}
In the case $n=3$, on a graph on a surface with a flat $\SLth$-connection $\Phi$,
one can use \emph{skein relations} (see Section \ref{ssec:websinsurfaces}) 
to write any $3$-multiweb as a linear combination of \emph{reduced} (or \emph{non-elliptic}) multiwebs.
These are multiwebs where each contractible face has $6$ or more sides (see precise definitions
in Section \ref{reductions}). Skein relations preserve the trace,
in the sense that the trace of the ``before" multiweb is the sum of the traces of the ``after" multiwebs.
We can thus rewrite the right-hand side of Theorem \ref{main} as a sum over isotopy classes $\lambda$ of reduced multiwebs:
$$\pm \det \tilde K(\Phi)
= \sum_{\lambda\in\Lambda_3}C_\lambda\Tr(\lambda).$$  
Here $\Lambda_3$ is the set of isotopy classes of reduced $3$-multiwebs.
By a theorem of Sikora and Westbury \cite{SikoraAlgGeomTop07} the coefficients $C_\lambda$ can be in principle extracted from $\det \tilde K(\Phi)$ as $\Phi$ 
varies over flat connections. See Sections \ref{annulussection} and \ref{pantssection} for applications.

\section*{Acknowledgements}  
	
We thank Charlie Frohman, James Farre, and David Wilson for helpful discussions, and the referees for suggestions for improvements. 
This research was supported by NSF grant DMS-1940932 and the Simons Foundation grant 327929.

\section{Background}

\subsection{Graph connections}
\label{ssec:slnlocalsystems}

Let $\Phi=\{\phi_e\}_{e\in E}$ be an \textit{$\SL(\R)$-local system} on $\G$.  
This is the data of, for each (oriented) edge $e=bw$ of $\G$, a matrix $\phi_{bw}\in\SL(\R)$, with $\phi_{wb} = (\phi_{bw})^{-1}$.  

A coordinate-free definition  can be given as follows. 
Let $V$ be a $n$-dimensional vector space, fixed once and for all, and associate to each vertex $v$ of $\G$ an identical
copy 
$V_v=V$; 
we call $\bigoplus_v V_v$ a 
\emph{$V$-bundle} on $\G$. A \emph{$\mathrm{GL}(V)$-connection} on $\bigoplus_v V_v$ is a choice of isomorphisms $\phi_{bw}:V_b\to V_w$
along edges $e=bw$, with $\phi_{wb}=\phi_{bw}^{-1}$.  Similarly we can talk about an \emph{$\mathrm{SL}(V)$-connection},
because $\phi_{bw}$ is a linear map from $V$ to itself.  Choosing a basis of $V$ allows us to talk about $\GL(\R)$- and  $\SL(\R)$-connections.   (In practice, one simply takes $V=\mathbb{R}^n$ at each vertex and assigns matrices to each oriented edge.)
Thus an $\SL(\R)$-connection is the same thing as an $\SL(\R)$-local system.  

More generally, we define an \emph{$\mathrm{End(V)}$-connection}
to be the assignment of a linear map $\phi_{bw}: V_b \to V_w$ to each edge $e=bw$,
not requiring invertibility (and we do not define any linear map from $w$ to $b$).  Choosing a basis $\beta$ of $V$ allows us to talk about an \emph{$\mathrm{M}_n(\R)$-connection.}

Two $\mathrm{End}(V)$-connections $\Phi=\{\phi_e\},\Phi'=\{\phi'_e\}$ on the same graph 
are \emph{$\mathrm{GL}(V)$-gauge equivalent} (resp. \textit{$\mathrm{SL}(V)$-gauge equivalent}) if there are $A_v\in\mathrm{GL}(V)$ (resp. $A_v \in \mathrm{SL}(V)$) such that for all edges $bw$,
we have $\phi'_{bw} = A_w^{-1}\phi_{bw}A_b$.  Similarly, we can talk about $\mathrm{M}_n(\mathbb{R})$-connections being $\mathrm{GL}_n(\mathbb{R})$- or $\mathrm{SL}_n(\mathbb{R})$-gauge equivalent.

Given 
a $\mathrm{GL}(V)$-connection $\Phi$ and a closed oriented loop $\gamma$ on $\G$ with a base vertex $v\in\gamma$, the \emph{monodromy} of $\Phi$ around $\gamma$  is the composition
of the isomorphisms around $\gamma$ starting at $v$. The conjugacy class of the monodromy is well-defined independently of the base point $v$.

\subsection{Dimer model}\label{dimermodel}

For background on the dimer model see \cite{Kenyon09Statisticalmechanics}.

Let $\G=(W\cup B,E) $ be a planar bipartite graph (graph embedded in the plane)
with the same number of white vertices as black vertices: $N = |W|=|B|$. Let $\nu:E\to\R_{>0}$ 
be a positive weight function on its edges.
A dimer cover of $\G$ is a perfect matching: a collection of edges covering each vertex exactly once. A dimer cover $m$ 
has weight $\nu_m$ given by the product of its edge weights:
$\nu_m = \prod_{e\in m}\nu_e$.  Let $\Omega_1$ be the set of dimer covers and let $Z=\sum_{m\in\Omega_1}\nu_m$ be the sum of weights of all dimer covers.
Kasteleyn showed \cite{Kasteleyn61Physica} that  $Z= |\det K|$,
where the matrix $K=(K_{wb})_{w\in W,b\in B}$, the (small) Kasteleyn matrix,  satisfies
$$K_{wb} = \begin{cases} \eps_{wb} \nu_{bw}& w\sim b\\0&\text{else}.\end{cases}$$
Here the $\eps_{wb}\in\{\pm1\}$ are signs chosen using the ``Kasteleyn rule": faces of length 
$l$ have $(l/2+1)\bmod 2$ minus signs.
This condition determines $K$ uniquely up to gauge, that is, up to left- and right-multiplication by a diagonal matrix of $\pm 1$'s.
See \cite{Kenyon09Statisticalmechanics}.

Later we will embed $\G$ in a planar surface, possibly with holes. In this setting we still call ``face" a connected component of the complement of the graph. A face is \emph{punctured} if it contains one or more holes, and otherwise is \emph{contractible}.
The Kasteleyn sign rule above applies for both punctured and contractible faces. 

Note that in our definition, $K$ has rows indexing white vertices and columns indexing black vertices.
Some references define the (big) Kasteleyn matrix $K$ indexed by all vertices, in which case $Z=|\det K|^{1/2}$.

\subsection{Double dimer model}\label{ddsection}

A double-dimer configuration on $\G$ is another name for a $2$-multiweb on $\G$. This is a decomposition of the graph into a collection of vertex-disjoint (unoriented) loops 
and doubled edges. If $\Phi=\{\phi_{bw}\}_{bw\in E}$ is an $\text{SL}_2(\mathbb{R})$-connection on $\G$,
and $m\in\Omega_2(\G)$ is a $2$-multiweb, we can compute the web-trace by
$$\Tr(m) =  \prod_{\text{loops $\gamma$}} \tr(\phi_{\gamma})$$
where $\tr$ is the matrix trace of the monodromy $\phi_\gamma$, and
where the product is over loops $\gamma$ of $m$ (each with some chosen orientation). 
Doubled edges do not count as loops
and so do not contribute to the product. Note that $\text{SL}_2$ has the special property that $\tr(A)=\tr(A^{-1})$,
so $\tr(\phi_{\gamma})$ is independent of the choice of orientation of $\gamma$.

In this setting we can construct an associated
Kasteleyn matrix $K=(K_{wb})_{w\in W,b\in B}$ as for single dimer covers, but with entries $K_{wb} = \eps_{wb}\phi_{bw}\in\M_2(\R)$, where $\eps_{wb}$ are 
Kasteleyn signs as before (and $K_{wb}$ is the zero matrix if $w$ and $b$ are not adjacent).
Note that we use $\eps_{wb}\phi_{bw}$ for the $K_{wb}$ entry, and not $\eps_{wb}\phi_{wb}$. This is 
because, having chosen $K$ to have white rows and black columns (rather than the reverse),
now both $\phi_{bw}$ and $K$ are maps from functions on black vertices to functions on white vertices. 
Let $\tilde K$ be the $2|W|\times 2|B|$ matrix obtained by replacing each entry with its
$2\times 2$ block of numbers. In \cite{Kenyon14CMP} it was shown that
\be\label{detK2}\det\tilde K=\sum_{m\in\Omega_2(\G)}\Tr(m).\ee

The main result of the present paper is to generalize this to $n\ge 3$. However even in the case $n=2$ our main theorem 
generalizes (\ref{detK2}) to $\M_2(\R)$-connections, where monodromies are no longer defined.

\section{Web-traces}
\label{sec:webtraces}

\subsection{Trace for proper multiwebs via tensor networks}
\label{ssec:defpropermultiwebs}

Let $\G$ be a bipartite graph, and 
let $\Phi=\{\phi_{bw}\}$ be an $\mathrm{End}(V)$-connection on $\G$ 
(or an $M_n(\R)$-connection, for $V=\mathbb{R}^n$, namely the assignment of an $n\times n$ matrix $\phi_{bw}$ 
to each edge of $\G$, not requiring invertibility).
We choose for each vertex $v$ a linear order of the half-edges at $v$. Denote by $L=L(\G)$ the data of these linear orderings.

A \emph{proper $n$-multiweb} in $\G$ is a $n$-multiweb $m\in\Omega_n(\G)$ whose multiplicities are all equal to $0$ or $1$, or in other words an $n$-valent spanning subgraph of $\G$.
Let $m$ be a proper $n$-multiweb in $\G$.
We consider its edges to be oriented from black to white.
The trace of $m$ is a scalar function of $\Phi$ associated to $m$, 
defined as follows  (the definition of trace for a general multiweb is given
in Section \ref{multiplicity} below).
This is a standard tensor network definition; see for example \cite{SikoraTrans01}.  

The data $L$ for $\G$ restricts to give a linear ordering of the $n$ half-edges of $m$ at each vertex, which by abuse of notation we also call $L=L(m)$.
Let $x_1,x_2,\dots,x_n$ be a basis of $V$, and $y_1,y_2,\dots,y_n$ the corresponding dual basis of the dual vector space $V^*$.
We associate identical copies $V_e=V$ and $V_e^*=V^*$ to each edge $e=bw$ of $m$.
Here $V_e$ is associated to the half-edge incident to $b$, and $V_e^*$ to the half-edge incident to $w$.
At each black vertex with adjacent vector spaces $V_1,V_2, \dots, V_n$ in the given linear order $L$
we associate a certain vector $v_b\in V_1\otimes \dots \otimes V_n$:
$$v_b = \sum_{\sigma\in S_n} (-1)^{\sigma} x_{\sigma(1)}\otimes x_{\sigma(2)} \otimes\dots\otimes x_{\sigma(n)},$$
called the \emph{codeterminant}.  
Likewise at each white vertex with adjacent vector spaces $V_1^*,\dots,V_n^*$ in order
we associate a dual codeterminant $u_w\in V_1^*\otimes \dots\otimes V_n^*$:
$$u_w = \sum_{\sigma\in S_n} (-1)^{\sigma} y_{\sigma(1)}\otimes y_{\sigma(2)} \otimes\dots\otimes y_{\sigma(n)}.$$


Now along each edge $e=bw$ we have a contraction of tensors using $\phi_{bw}$. 
That is, we take the tensor product 
$$\bigotimes_{b} v_b \in \bigotimes_{e} V_e$$ 
of $v_b$ over all black vertices, and the tensor product
$$\bigotimes_{w} u_w \in \bigotimes_e V^*_e$$ of $u_w$ over all white vertices.
Then we contract component by component along edges:
a simple tensor
$\bigotimes_{e} x_e \in \bigotimes_{e} V_e$ and a simple tensor $\bigotimes_{e} y_e\in \bigotimes_e V^*_e$ contract to give 
$\prod_{e=bw} y_e(\phi_{bw}x_e),$
or, in a more symmetric notation $\prod_{e=bw}\langle y_e|\phi_{bw}|x_e\rangle$.
Summing we have
$$\Tr_L(m, \Phi) :=  \left\langle \bigotimes_{w} u_w \Big|\bigotimes_{e=bw}\phi_{bw}\Big|\bigotimes_bv_b\right\rangle.$$

The above definition of trace $\Tr_L(m,\Phi)=\Tr_L(m,\Phi)(x_i)$ was calculated in terms of the basis $\{x_i\}$ for $V$, but in fact does not depend on this choice.
That is, if $\{ x_i^\prime\}$ is another basis of $V$, then $\Tr_L(m,\Phi)(x_i)=\Tr_L(m,\Phi)(x_i^\prime)$.
 Indeed, under a base change, the codeterminants $v_b\in V_1\otimes \dots \otimes V_n$ multiply by the determinant
of the base change matrix, and the dual codeterminants $u_w\in V_1^*\otimes \dots\otimes V_n^*$ multiply by the inverse of this determinant.

\subsection{Definition of trace for general multiwebs}\label{multiplicity}
\label{ssec:multiwebtraces}

Equip $\G$ with a linear order $L=L(\G)$ as in Section \ref{ssec:defpropermultiwebs}.

Now let $m\in\Omega_n(\G)$ be a general multiweb in $\G$; see Section \ref{ssec:multiwebs}.  
The data $L$ restricts to a linear order $L=L(m)$ of the (possibly fewer than $n$) half-edges of $m$ incident to each vertex.  
We define the trace of $m$ with respect to an $\mathrm{End}(V)$-connection $\Phi$ of $\G$ as follows. Split each edge $e$ of $\G$ of multiplicity $m_e>1$ into $m_e$ parallel edges (edges with the same endpoints); let $\G'$ be the resulting
 graph. If the original edge $e=bw$ has parallel transport $\phi_{bw}$, put $\phi_{bw}$ on each of the new split edges; call the new connection on $\G^\prime$ by $\Phi^\prime$.
The multiweb $m$ in $\G$ 
determines a proper multiweb $m^\prime\in\Omega_n(\G^\prime)$, where an edge of multiplicity $m_e$ of $m$ 
 becomes $m_e$ parallel edges of multiplicity $1$ in $m'$.  
 
Refine the ordering data $L$ of $m$ to an ordering data $L^\prime$ of $m^\prime$ by choosing at each black vertex an arbitrary linear order for the new split half-edges, while still respecting the overall $L$ order at that vertex;
this 
determines a corresponding ordering of the split half-edges at the adjacent white vertex.
We define 
\begin{equation}\label{mwdef}\Tr_L(m, \Phi) := \frac{\Tr_{L^\prime} (m^\prime, \Phi^\prime)}{\prod_e \left( m_e! \right)}\end{equation}
where the trace $\mathrm{Tr}_{L^\prime}(m^\prime, \Phi^\prime)$ of the proper $n$-multiweb $m^\prime$ is defined as in Section~\ref{ssec:defpropermultiwebs} above.  Notice the trace $\Tr_L(m, \Phi)$ is independent of the choice of linear order $L'$ refining $L$:
any other such linear order $L^\prime $ will change each codeterminant and adjacent dual codeterminant by the same signature.

Proposition \ref{tracecolor0} below justifies this definition.

\subsection{\texorpdfstring{$\mathrm{SL}_n$}{SLn} gauge equivalence}

Recall the notion of two 
$\mathrm{End}(V)$-connections $\Phi$, $\Phi^\prime$ being 
$\mathrm{GL}(V)$- or 
$\mathrm{SL}(V)$-gauge equivalent; see Section \ref{ssec:slnlocalsystems}.  One can check 
that $\Tr_L(m,\Phi)\neq\Tr_L(m,\Phi^\prime)$ in general if $\Phi$ and $\Phi^\prime$ are only 
$\mathrm{GL}(V)$-gauge equivalent.  However, $\mathrm{SL}(V)$ gauge equivalence preserves the trace:

\begin{prop}\label{indep} The trace of an $n$-multiweb $m\in\Omega_n(\G)$ is 
$\mathrm{SL}(V)$-gauge invariant.  That is, $\Tr_L(M, \Phi)=\Tr_L(M, \Phi^\prime)$ if the 
$\mathrm{End}(V)$-connections $\Phi$ and $\Phi^\prime$ are 
$\mathrm{SL}(V)$-gauge equivalent.
\end{prop}

\begin{proof}
An $\SL$ change of basis at a single black vertex $b$ preserves the codeterminant $v_b$. Likewise at a single white vertex.
\end{proof}

\subsection{Traces in terms of edge colorings}

The trace of a multiweb $m$ can be given a more combinatorial definition as follows.
At each vertex $v$, choose a coloring of the half-edges of $m$ at $v$ with $n$ colors $\CC=\{1,\dots,n\}$, using each color once,
and so that  
half-edges with multiplicity get a subset of colors. In other words at a vertex $v$ with edge multiplicities $m_1,m_2,\dots,m_k$, 
we partition $\CC$ into $k$ disjoint 
subsets of sizes $m_1,m_2,\dots,m_k$, one for each half-edge at $v$.
Note that there are $\binom{n}{m_1,m_2,\dots,m_k}$ possible such colorings at $v$.  

To such a coloring at $v$ we associate a sign $c_v$, defined as follows.  
Let the color set $\CC$ be equipped with the natural order, and let $L$ be the linear order of half-edges at $v$ as in
Section \ref{ssec:multiwebtraces}.  List the colors at $v$ according to the linear order $L$ where 
for edges with multiplicity $>1$ the colors are listed in their natural order. Then this list is a permutation of $\CC$ and $c_v$ is its signature.

A \emph{coloring} $c$ of the multiweb $m$ is the data of a coloring of the half-edges at each vertex $v$. Given a coloring of $m$, on each edge $e$ of multiplicity $m_e$ there are two sets of colors, each of size $m_e$, 
one associated with the black vertex and one with the white vertex.
We associate a corresponding matrix element to this edge $e$, using the bijection of colors with indices: 
if $e=bw$ has parallel transport $\phi_{bw} \in \mathrm{M}_n(\mathbb{R})$, and is colored with set $S_e$ at the white vertex and $T_e$ at the black vertex, then the corresponding matrix element is $\det (\phi_{bw})_{S_e,T_e}$, that is, the determinant of the $S_e,T_e$-minor of $\phi_{bw}$ (the submatrix of $\phi_{bw}$ with rows $S_e$ and columns $T_e$).  Recall that the rows (resp. columns) of $\phi_{bw}$ are indexed by the colors at $w$ (resp. $b$). 

Now to each coloring $c$ of $m$ we take the product of the associated matrix elements over all edges, and multiply this by the sign
$\prod_v c_v$. Then this quantity is summed
over all possible colorings to define the trace:

\begin{prop}\label{tracecolor0} The above procedure computes the trace of the $n$-multiweb $m$, that is,
\be\label{trline0}\Tr_L(m, \Phi) = \sum_{\textnormal{colorings $c$}} \prod_v c_v\prod_{e=bw}\det(\phi_{bw})_{S_e,T_e}.\ee  
\end{prop}

\begin{proof}
If $m$ is a proper multiweb, that is, if all edges have multiplicity $1$, then a coloring of the half-edges at $b$ corresponds to a single term
in $v_b$, the codeterminant, and its sign is the coefficient in front of this term. So a coloring of all the half-edges at all vertices
corresponds to a single term in the expansion of 
\be\label{trdef}\Tr_L(m) = \langle \bigotimes_wu_w |\bigotimes_{e=bw}\phi_{bw}|\bigotimes_{b} v_b\rangle\ee
when we expand $ \otimes_{b} v_b$ and $ \otimes_{w}u_w$ over all permutations.
Thus the two formulas (\ref{trdef}) and (\ref{trline0}) agree for proper multiwebs.

Now suppose $m$ is a multiweb, obtained from a proper multiweb $m'$ by collapsing $k$ parallel edges, each with parallel transport $\phi_{bw}$, into a single
edge  $e=bw$ of multiplicity $k$, with parallel transport $\phi_{bw}$. 
In $m$ a coloring assigns subsets $S_e,T_e\subset\CC$ of size $k$ to $e$, and the contribution from this edge is $\det (\phi_{bw})_{S_e,T_e}$ for the multiweb $m$.
The corresponding contribution in $\Tr_L(m')$ 
from this set of colorings involves all possible ways of distributing the colors $S_e$ to the $k$ half edges at $w$, and likewise for $T_e$.
There are $k!$ bijections of $S_e$ with $e_1,\dots,e_k$, and $k!$ bijections of $T_e$ with $e_1,\dots,e_k$. 
Each such choice is a term in $\det (\phi_{bw})_{S_e,T_e}$, and there are $k!$ choices corresponding to each term; moreover the signs agree. Thus the Proposition holds for $m$, using the definition
$\Tr_L(m') = k!\Tr_L(m)$ from (\ref{mwdef}). Splitting any other multiple edges, we argue analogously. This completes the proof.
\end{proof}

If the connection is the \emph{identity connection}, we have a nonzero contribution to the trace only when for each edge $e$,  $S_e=T_e$, that is,
the subsets of colors for the two half-edges are equal. Such a coloring of $m$ is called an \emph{edge-$n$-coloring}.  In this case the matrix element is exactly $1$ for each edge.
The trace is thus the signed number of edge-$n$-colorings,
where the sign is $\prod_v c_v$. 

\begin{prop}\label{tracecolor}
The trace $\Tr_L(m,I)$ of an $n$-multiweb for the identity connection is the signed number of its edge-$n$-colorings.
\end{prop}

A similar proposition holds for diagonal connections (where a diagonal matrix is assigned to each edge), where now the trace is a weighted, signed number of edge-$n$-colorings.

\subsection{Graphs in surfaces}
\label{ssec:graphsinsurfaces}

Let $\G$ be a bipartite graph embedded in an oriented surface $\Sigma$.
The half-edges of $\G$ incident to a vertex $v$ inherit a cyclic ordering from the orientation of $\Sigma$.  Our convention is to choose the counterclockwise (resp. clockwise) orientation when $v$ is a black (resp. white) vertex.  This cyclic ordering can be upgraded to a linear ordering by choosing a ``starting'' half-edge at $v$; in pictures, we indicate this choice of preferred half-edge by drawing a \emph{cilium} emanating from the vertex; see Figure \ref{ciliafig} and \cite{FockAmerMathSocTransl}. We denote by $L=L(\G)$ the data of a choice of cilia at the vertices.  

We suppose $\G$ is equipped with an $\mathrm{End}(V)$-connection $\Phi$. 
Now for a multiweb $m\in\Omega_n(\G)$,
the trace $\Tr_L(m,\Phi)$ is defined as in Section \ref{ssec:multiwebtraces}.

\begin{figure}
\begin{center}\includegraphics[width=2in]{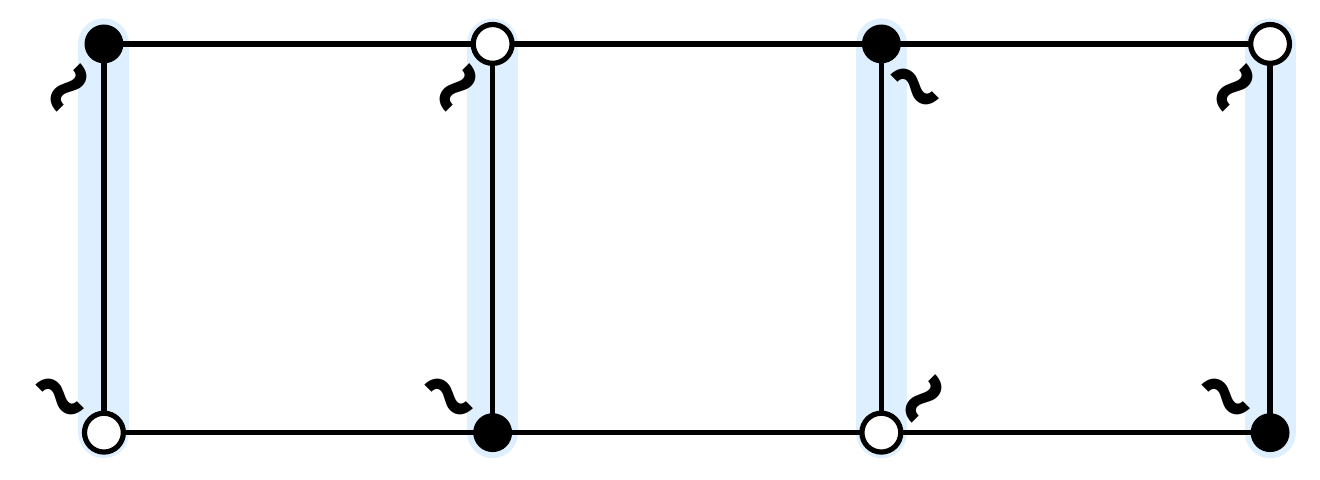}
\end{center}
\caption{\label{ciliafig} Graph $\G$ with ``positive'' cilia $L$ (see Corollary \ref{poscilia}).}
\end{figure}

\begin{prop}\label{click}
	If two cilia data $L$ and $L^\prime$ differ by rotating a single cilium by one ``click'', then \begin{equation*}\Tr_L(m,\Phi)=\pm \Tr_{L^\prime}(m,\Phi)\end{equation*}
for all $m \in \Omega_n(\G)$, where the sign is $+$ if $n$ is odd and $(-1)^{m_e}$ if $n$ is even, where $e$ is the edge of $\G$ crossed upon rotation of the cilium.
\end{prop}
\begin{proof}
	Proposition \ref{tracecolor0}.  
\end{proof}

It follows from the proposition that web-traces are independent of the choice of cilia for $n$ odd.

\subsection{Planar graphs}\label{planarwebscn}

Let $\G$ be embedded in an oriented \emph{planar}  surface $\Sigma$, that is, in a genus-$0$ surface minus $k$ disjoint closed disks, $k \geq 0$.  Equip $\G$ with a choice of cilia $L$, as described in the previous section.  

Proposition \ref{tracecolor} simplifies for planar multiwebs:
\begin{prop}\label{tracecolorplanar}
The trace $\Tr_L(m,I)$ of a planar $n$-multiweb $m\in\Omega_n(\G)$ for the identity connection is equal to $\pm1$ times the number of edge-$n$-colorings (with sign $+$ if $n$ is odd).  This is also nonzero, by Section \ref{finding} below.
\end{prop}

\begin{proof}  
First note that given an edge-$n$-coloring $C$ of a multiweb $m$, the set of edges containing color $i$ is a dimer cover of $m$:
each vertex has exactly one adjacent edge containing color $i$. We will use the fact (originally due to Thurston \cite{Thurston90AmerMathMonthly}) 
that two dimer covers on a planar bipartite graph can be connected by a sequence of ``face moves" as in Figure \ref{facemove}.
\begin{figure}
\begin{center}\includegraphics[width=3in]{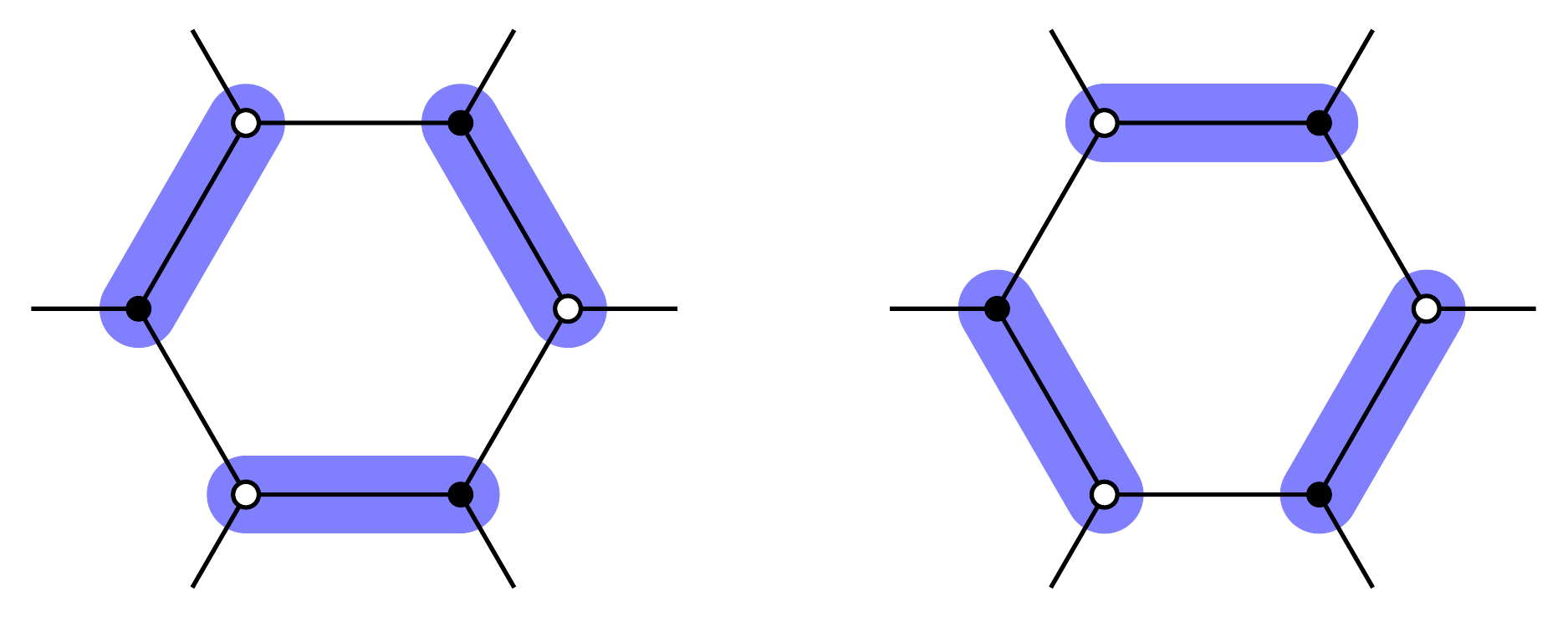}
\end{center}
\caption{\label{facemove} Suppose every other edge of a face is occupied by a dimer; one can shift all dimers around that face
to get a new dimer cover. This is the \emph{face move}. Such moves connect the set of all dimer covers for any 
planar bipartite graph which is connected and nondegenerate (every edge is used in some dimer cover). For general planar bipartite graphs a slight generalization of these works.}
\end{figure}

We now show that two edge-$n$-colorings $C,C'$ have the same sign. 
The idea is to move the color-$1$
edges of $C$ to those of $C'$ using face moves.
Then move the color-$2$ edges using face moves, and so on, until $C=C'$.

The proof is different when $n$ is odd and $n$ is even. Suppose first that $n$ is odd.

When we rotate color-$1$ edges around a face $f$, keeping the other colors fixed, 
we create a new colored multiweb $m^\prime$ with color $C^\prime$ (that is, $m^\prime$ has different edge multiplicities 
than $m$). 
Let us compute how the sign changes.
Suppose $f$ has vertices $b_1,w_1,\dots,b_k,w_k$ in counterclockwise order,
and edges $b_iw_i$ have multiplicities $M_i$.    
Since $n$ is odd, the trace is independent of the choice of cilia, so we can rotate the cilia at vertices of $f$ so that at each vertex of $f$ the cilia lies in $f$, that is,
the linear order starts inside the face $f$.
Then the first edge out of $b_i$ is $b_iw_{i-1}$ and the first edge out of $w_{i-1}$ is $w_{i-1}b_{i}$.  (Note here we have used that black vertices are oriented counterclockwise, and white vertices clockwise.)
Suppose in $C$ edge $b_iw_{i-1}$ contains color $1$. When we rotate this color
so that edge $b_iw_i$ now has color $1$,
in the linear ordering of colors at $b_i$ color $1$ has moved from the beginning to position $M_i$ from the end
($M_i$ edges separate it from the cilium).  
Likewise
at $w_{i-1}$, color $1$ has moved from the beginning to position $M_{i-1}$ from the end.   The net sign change after the face move,
which involves all color-$1$ edges of the face, is $(-1)^{\sum_{i=1}^k(M_{i-1}+M_{i})}=1$ where $M_0:=M_k$. The sign is thus preserved. 

Likewise for the other colors: changing the ordering of the colors does not change the sign (since 
$\G$ has an even number of vertices) 
so any other color can be considered
to be the first color, and then the same argument holds.

The global sign can be determined by transforming all dimer covers to the same dimer cover, so that the multiweb
is a union of $n$-fold edges. In this case the sign is $+$, since at each vertex the sign is $+$.  This completes the proof in the case $n$ is odd.

Now assume $n$ is even. The argument works as for the $n$ odd case if all cilia of vertices on a face $f$ are located 
in the face $f$. However
we also pick up a sign change from rotating the cilia. When we rotate the cilium at a vertex on $f$ to move it into face $f$, we pick up a sign change
depending on the parity of the total multiplicity of the edges we cross (since $n$ is even, moving it clockwise or counterclockwise results in the same sign change). Then, after rotating color $1$ around $f$, we rotate the cilium
back to its original position; we get another sign change. However the multiplicities of edges being crossed have changed by exactly $1$, so the new sign change is the opposite of the first sign change, and so the net sign change is $-1$. Thus rotating color $1$ at face $f$ gives a net sign change of $(-1)^{\text{out cilia}}$,
that is, $-1$ per cilia of vertices on $f$ which are not pointing into $f$. 

Now given two colorings of the same multiweb $m$, we claim that when we transform one to the other doing face moves,
each face move is performed an even number of times, that is, each face is toggled an even number of times. 
To see this, take a path $\gamma_f$ 
in the dual graph from $f$ to the exterior face.  Compute the $\Z/2\Z$-intersection number $X_f=\langle\gamma_f,m\rangle$ 
of $\gamma_f$ with the multiweb $m$ counting multiplicity (note this does not depend on the coloring of $m$).
Each $f$-face move changes $X_f$ by $1\bmod 2$, and each face move by any other faces changes $X_f$ by $0\bmod 2$. 
Any sequence of face moves changing $m$ back to $m$, but with a possibly different coloring, necessarily does not change $X_f$.  We don't ever need to toggle the outer face.

This completes the proof. 
\end{proof}

From the above proof we see that we can arrange the cilia so that (for $n$ even and the identity connection) all traces are positive: 

\begin{cor}\label{poscilia} For $n$ even and the identity connection, if cilia $L$ are chosen so that each face of $\G$ has an even number of cilia, then 
for any multiweb $m$ in $\G$, $\Tr_L(m,I)$ is equal to the number of edge-$n$-colorings.
\end{cor}
\begin{proof}
Under such a choice of cilia, when we perform a rotation at a face, the net sign does not change. Thus every edge-$n$-coloring of every multiweb has the same sign. A multiweb where 
all edges have multiplicity $n$ or $0$ has positive trace.
\end{proof}

We call a choice of cilia $L$ \emph{positive} if for the identity connection it yields positive traces for all multiwebs, and we write $\Tr_+$ instead of $\Tr_L$. 
One way to choose cilia so that each face has an even number (and is thus positive by the corollary) is as follows. Let $m_1$ be a dimer cover of $\G$.
For each edge of $m_1$, choose the cilia at the vertices at its endpoints to be both on the same face containing one of the two sides of $e$. See Figure \ref{ciliafig}

\subsubsection{Finding an edge-\texorpdfstring{$n$}{n}-coloring of a multiweb}
\label{finding}

Any regular bipartite planar graph has a dimer cover \cite{LP}; from this it follows by induction that any $n$-multiweb has an edge-$n$-coloring (an edge-$n$ coloring is a union of $n$ dimer covers).

However
there is a simple algorithm, communicated to us by Charlie Frohman, for finding an explicit edge-$n$-coloring of an $n$-multiweb of a planar graph,
as follows; see also \cite{FrohmanJKnotTheoryRam19}.  
Given a fixed multiweb $m$, define a height function on faces taking values in $\Z/n\Z$ as follows.
On a fixed face $f_0$ define $h(f_0)=0$.  
For any other face $f$
let $\gamma$ be a path in the dual graph from $f_0$ to $f$. The height change along every edge of $\gamma$ is given by the algebraic
intersection number of the common edge with $\gamma$ (considering edges oriented from black to white, and taking multiplicity in $m$ into account). 
In other words given two adjacent faces $f_1,f_2$ with edge $bw$ between them of multiplicity $k$, 
where $f_1$ is on the left, then $h(f_1)-h(f_2) = k$. Note that $h$ is well-defined since 
it is well-defined around each vertex.  

Given $h$, the color set of an edge of multiplicity $k$ is the interval of colors $\{a,a+1,\dots,a+k-1\}$ if the face to its right has height $a$.

\subsection{Trace examples}

\subsubsection{Theta graph}
\label{sssec:thetagraph}

Consider the following example for $\SLth$-connections (as opposed to $\mathrm{M}_3$-connections).
Let $m$ be a ``theta'' graph consisting of two vertices $w,b$ with three edges between them $e^1,e^2,e^3$, each of multiplicity $1$,
in counterclockwise order at $b$ and clockwise order
at $w$ (with respect to a planar embedding of the graph). Let $A,B,C$ be the parallel transports from $b$ to $w$
along edges $e^1,e^2,e^3$ respectively (recall that we always orient edges from black to white). We label the basis vectors using three colors, red, green, blue: $e_r,e_g,e_b$. (As a word of caution, the ``$b"$ refers here to the color blue; the ``$b$" on a vertex, e.g. $v_b$, refers to the color black.)
Then
$$v_b=e_r\otimes e_g\otimes e_b - e_r\otimes e_b\otimes e_g+ \dots -e_b\otimes e_g\otimes e_r$$
and
$$u_w=f_r\otimes f_g\otimes f_b - f_r\otimes f_b\otimes f_g+ \dots -f_b\otimes f_g\otimes f_r.$$

The contraction contains $36$ terms; for example, contracting the first terms
of $v_b$ and $u_w$ gives $A_{rr}B_{gg}C_{bb}$, and contracting the first term of $v_b$ and the second term of $u_w$ gives $-A_{rr} B_{bg} C_{gb}$. Summing, the trace can be written compactly as
$$\Tr(m) = \tr(AB^{-1})\tr(CB^{-1})- \tr(AB^{-1}CB^{-1})$$ (see Section \ref{sssec:basicskein} below)
or, more symmetrically, as
\be\label{xyz}[xyz]\det(xA+yB+zC),\ee
that is, the coefficient of $xyz$ in the expansion of the determinant of $xA+yB+zC$ as a polynomial in $x,y,z$
(see (\ref{xyzthm}) below); note \eqref{xyz} is valid for any $\mathrm{M}_3$-connection.

Also notice that when $A=B=C$ then the trace $\Tr(m)=+6$; in this case the $\mathrm{SL}_3$-local system is trivializable, and the trace is thus the number of edge-$3$-colorings (see Proposition \ref{tracecolorplanar}). 

More generally, for the $\mathrm{M}_n$-connection on the  $n$-theta graph $m$ (two vertices with $n$ edges) consisting of the same matrix $A \in \M_n$ attached to each edge, by Proposition \ref{tracecolor0} the trace is $\Tr(m) = \pm n! \mathrm{det}(A)$.

As one last variation, again in the case $n=3$, note that if the cyclic ordering of the vertices is that induced by the nonplanar embedding of the theta graph in the torus as in Figure \ref{theta}, then its trace $\Tr(m)$ with respect to the identity connection is equal to $-6$ (since we just reverse the cyclic order at one of the two vertices); contrast this with the calculation above and Proposition \ref{tracecolorplanar}.
\begin{figure}[htbp]
\center{\includegraphics[width=1in]{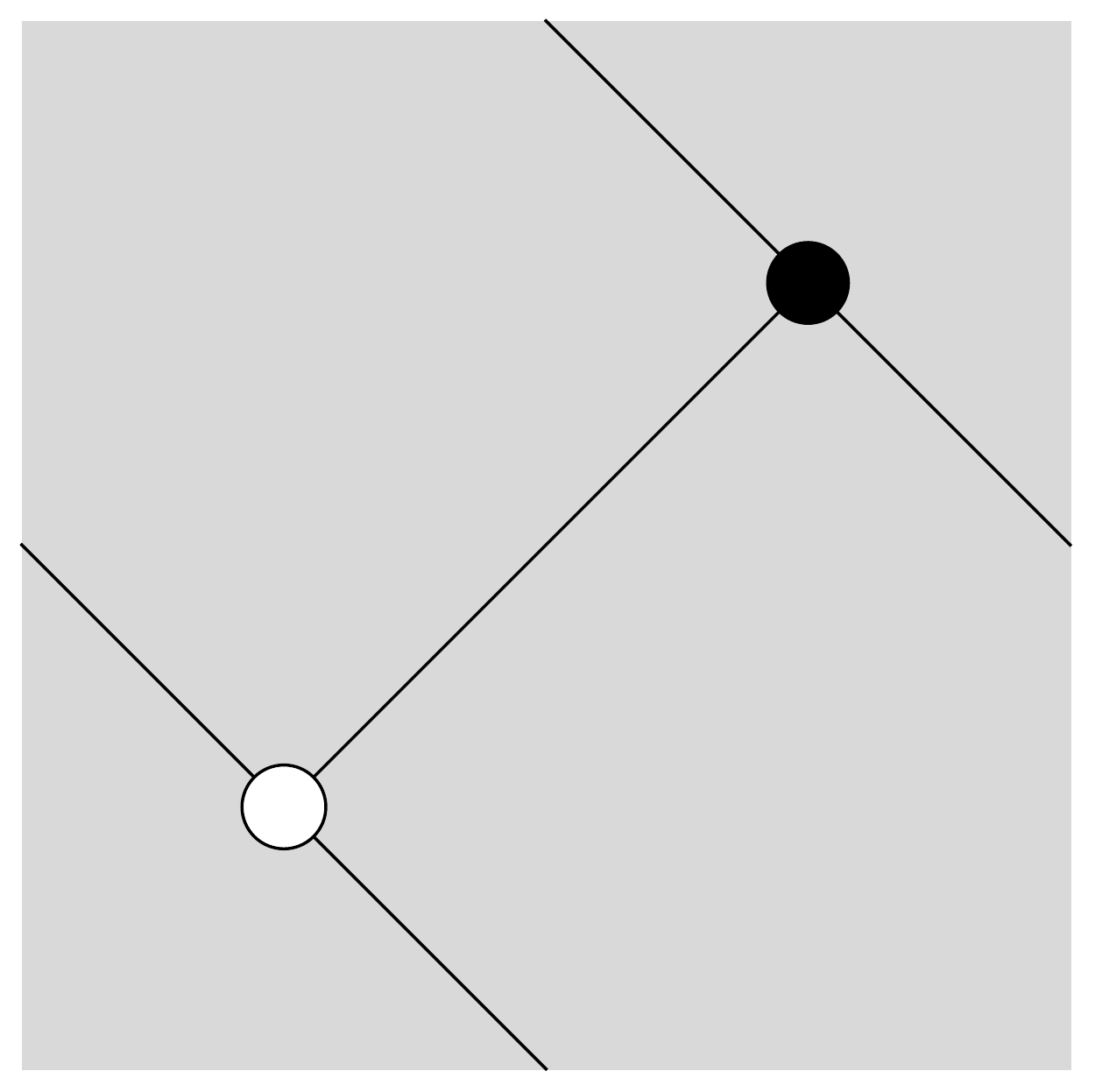}}
\caption{\label{theta}The theta graph embedded in the torus}
\end{figure}

\subsubsection{Loop}\label{loop}

As another example with $n=3$, let $m$ consist of a cycle $b_1w_1b_2w_2$ with edges $b_{i}w_{i}$ of multiplicity $1$, and edges $b_iw_{i-1}$ of multiplicity $2$, as in Figure \ref{4cycle}. 
\begin{figure}[htbp]
\center{\includegraphics[width=1in]{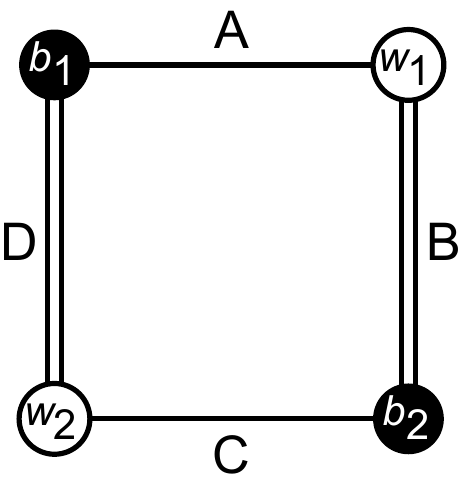}}
\caption{\label{4cycle}A $4$-cycle in a $3$-multiweb.}
\end{figure}
At the black vertex $b_1$, say, 
$$v_{b_1} = \sum_\sigma(-1)^{\sigma} e_{\sigma(1)}\otimes e_{\sigma(2)}\otimes e_{\sigma(3)}=
(e_r\wedge e_g)\otimes e_b +(e_g\wedge e_b)\otimes e_r +(e_b\wedge e_r)\otimes e_g.$$ 
Here we have used the wedge notation $e\wedge e':= e\otimes e' - e'\otimes e \in \R^3 \otimes \R^3$. 
\
Choosing in addition an $\mathrm{M}_3$-connection as in the figure, starting at $b_1$ this leads to 
\begin{equation}\label{eq:looptrace}  \Tr(m) = \tr(
\widetilde{D \wedge D}\cdot C\cdot 
\widetilde{B\wedge B}\cdot A).\end{equation} 
For an $\SLth$-connection, this becomes $\Tr(m)=\tr(D^{-1}CB^{-1}A)$, namely the trace of the monodromy when the cycle is oriented clockwise in the figure. In particular, note that
even though the cycle is not naturally oriented, the $3$-web-trace nevertheless picks out an orientation: the one determined by following from black to white along the non-doubled edges.

Equation \eqref{eq:looptrace} needs to be 
interpreted, where $
\widetilde{B \wedge B} \in \mathrm{M}_3(\mathbb{R})$ is defined as follows:  Let $
B \wedge B \in \mathrm{M}_3(\mathbb{R})$ be the matrix of the linear map $
\Lambda^2 \mathbb{R}^3 \to \Lambda^2 \mathbb{R}^3$ induced by $B$ written in the 
ordered basis $\left\{ e_r \wedge e_g, e_r \wedge e_b, e_g \wedge e_b \right\}$, and let $\varphi = \left( \begin{smallmatrix} 0&0&1\\0&-1&0\\1&0&0 \end{smallmatrix} \right)$.  Then $
\widetilde{B \wedge B}  := ( \varphi \circ 
(B \wedge B) \circ \varphi^{-1} )^t$.
Here $\varphi$ induces an isomorphism 
 $
\Lambda^2 \mathbb{R}^3 \cong 
(\Lambda^1 \mathbb{R}^3)^*= (\mathbb{R}^3)^*$, and the transpose in the formula for $
\widetilde{B \wedge B}$ allows one to pass from the dual $(\mathbb{R}^3)^*$ to $\mathbb{R}^3$.  It also corresponds to the transpose in the formula for the inverse of $B$ in terms of its cofactor matrix, so that $
\widetilde{B \wedge B} = B^{-1}$ when $B \in \mathrm{SL}_3(\mathbb{R})$; hence, we have the formula above for the trace $\Tr(m)$ for $\mathrm{SL}_3$-connections.

For general $n$, consider a cycle $b_1w_1b_2w_2\dots  b_\ell w_\ell$ with edges
$b_{i}w_{i}$ of multiplicity $k$, and
edges $b_iw_{i-1}$ of multiplicity $n-k$, and an $M_n$-connection with parallel transport $A_i$ from $b_i$ to $w_{i}$, and $B_{i-1}$ from $b_{i}$ to $w_{i-1}$.  Then, we have 
$\Tr_L(m) = \pm\tr(C)$, where $C$ is the product $$C=(
\widetilde{\wedge^{n-k}B_\ell})(\wedge^{k} A_\ell) \cdots (
\widetilde{\wedge^{n-k} B_2}) (\wedge^{k}A_2)(
\widetilde{\wedge^{n-k}B_1})(\wedge^{k}A_1)$$
appropriately interpreted, similar to the case $n=3$, using the isomorphism 
$\Lambda^{n-k} \mathbb{R}^n \cong (\Lambda^k \mathbb{R}^n)^*$.  For instance, 
$$\varphi=\left(\begin{smallmatrix}
0&0&0&0&0&1\\
0&0&0&0&-1&0\\
0&0&0&1&0&0\\
0&0&1&0&0&0\\
0&-1&0&0&0&0\\
1&0&0&0&0&0
\end{smallmatrix}\right)$$ in the case $n=4$, $k=2$, \and with respect to the basis $\left\{ e_{i_1} \wedge e_{i_2}; 1 \leq i_1 < i_2 \leq 4 \right\}$ ordered lexicographically.

For an $\SL$-connection with total monodromy $M=B_\ell^{-1} A_\ell \cdots B_1^{-1} A_1$
clockwise around the loop, we thus have
$$\Tr_L(m) = \pm\tr(\wedge^k M).$$
Indeed, by first taking advantage of the $\mathrm{SL}_n$-gauge invariance of the web-trace (Proposition \ref{indep}) to concentrate the connection $M$ entirely on the edge $b_1 w_1$, this is then an immediate consequence of the above equation. We remind that for $n$ odd the sign is $+$ but for $n$ even the sign depends on the choice of cilia~$L$.  Note that for all $n$, when $k=n$, corresponding to when $m$ consists of $\ell$ disjoint $n$-multiedges,  we obtain $\Tr_L(m)=+\mathrm{det}(M)=1$.

\subsubsection{Nonplanar example}

Consider the complete bipartite graph on three vertices $m=K_{3,3}$.  This is nonplanar.  For the cyclic ordering on the vertices induced by the embedding of $m$ in the torus as shown in Figure \ref{K33}, one computes $\Tr(m, I_3)=0$;  
\begin{figure}
\begin{center}\includegraphics[width=1in]{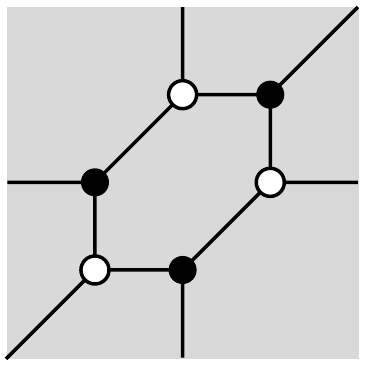}
\end{center}
\caption{\label{K33}An embedding of the graph $K_{3,3}$ in the torus}
\end{figure}
 for example, one can use either the definition or the basic skein relation (Section \ref{sssec:basicskein}).  On the other hand, $m$ has $12$ edge-$3$-colorings.  This demonstrates that Proposition \ref{tracecolorplanar} is special to planar graphs.

\section{Kasteleyn matrix}

Let $\G$ be a bipartite planar graph (see Section \ref{planarwebscn}), with the same number $N$ of black vertices and white vertices.
We fix a choice of positive cilia $L$ for $\G$, in the sense of Corollary \ref{poscilia} and its subsequent paragraph.
Let $\Phi=\{\phi_e\}_{e\in E}$ be an $\M_n(\R)$-connection on $\G$.  
	
Let $K\in \M_N(\M_n(\R))$ be the associated Kasteleyn matrix:
$K$ has rows indexing white vertices and columns indexing black vertices, with entries $K(w,b)=0_n$ (the $n\times n$ zero matrix) if $w,b$ are not adjacent and otherwise
$K(w,b) = \eps_{wb} \phi_{bw}$,
where the signs $\eps_{wb}\in\{-1,+1\}$ are given by the Kasteleyn rule (see Section \ref{dimermodel}). For the purposes of
defining Kasteleyn signs we impose the sign condition on cycles of $\G$ bounding each complementary component (not just contractible ones). 
We let $\tilde K=\tilde K(\Phi)$ be the $nN\times nN$ matrix obtained from $K$ by replacing each entry
with its $n\times n$ array of real numbers.

Note there is also the obvious coordinate-free description.  Given an $\mathrm{End}(V)$-connection $\Phi$ on $\G$, the Kasteleyn determinant $\det \tilde{K}(\Phi)$, depending on a choice of Kasteleyn signs for $\G$ (see Section \ref{dimermodel})  and an ordering of the black and white vertices, is defined as the determinant of the induced linear endomorphism $V^{|B|} \to V^{|W|}$.   

Note
that $\mathrm{det} \tilde{K}(\Phi)=\mathrm{det} \tilde{K}(\Phi^\prime)$ if $\Phi$ and $\Phi^\prime$ are $\SL$-gauge equivalent.  (The gauge group $(\SL)^V$ acts on $K$ by left- and right-multiplication by diagonal matrices with entries in $\SL$,
or equivalently, acts on $\tilde K$ by left- and right-multiplication by block-diagonal matrices with blocks in $\SL$.
These gauge equivalences do not change the determinant of $\tilde K$.) 
It is not true in general that $\det \tilde{K}(\Phi)=\det \tilde{K}(\Phi^\prime)$ if $\Phi$ and $\Phi^\prime$ are only 
$\GL$-gauge equivalent.

Recall that $\Omega_n(\G)$ is the set of $n$-multiwebs in $\G$.  
Our main theorem is
\begin{thm}\label{main}  Let $\Phi$ be an $\mathrm{End}(V)$-connection on the planar bipartite graph~$\G$.

For $n$ even, and a choice of positive cilia $L$ as discussed above,
\be\label{maineq1}
	\det \tilde K(\Phi) = \sum_{m \in \Omega_n(\G)} \Tr_+(m).
\ee

For $n$ odd, 
\be\label{nodd}
	\!\!\!\!\!\!\!\!\!\!\!\pm \det \tilde K(\Phi) = \sum_{m \in \Omega_n(\G)} \Tr(m).
\ee
\end{thm}

Note for $n$ odd, we can write $\Tr$ in (\ref{nodd}) without reference to a choice of cilia since $\Tr_L$ is independent of $L$ by Proposition \ref{click}.

Note for $n$ even, one can check that for arbitrary (rather than positive) choices of cilia $L$ the two sides of (\ref{maineq1})
are not generally equal even up to sign.

The theorem implies that for trivial connections, and $n$ even, the determinant of the Kasteleyn matrix of a planar bipartite graph admitting a dimer cover is always strictly greater than zero. In particular, it does not depend on the choice of Kasteleyn signs.

Also note that for $n=2$ the theorem gives a new and different proof of a more general version of Theorem 2 of \cite{Kenyon14CMP}.

For our definition of $\tilde K$ the sign ambiguity is inevitable when $n$ is odd since the sign of 
$\det\tilde K$ depends on an arbitrary choice of order for both white vertices and black vertices.   (It also depends on the choice of Kasteleyn signs, unlike the $n$ even case.)

\begin{proof}  
We assume for notational  
clarity that $\G$ is a simple graph; if $\G$ has multiple edges between pairs of vertices, a slight variation of the following proof will hold. 

Let $\G_n$ be the graph obtained from $\G$ by replacing each vertex $v$ with $n$ copies $v^1,\dots,v^n$,
and replacing each edge $bw$ with the complete bipartite graph $K_{n,n}$ connecting the $b^j$ and $w^i$.  See Figure \ref{proof}.
For edge $e=bw$ of $\G$, let $A_{bw}=\eps_{wb}\phi_{bw}$ be the parallel transport $\phi_{bw}$ times the Kasteleyn sign $\eps_{wb}$.
For $i,j\in\CC=\{1,2,\dots,n\}$ put weight $A_{bw}^{ij}:=(A_{bw})_{ij}$ on the edge $b^jw^i$ of $\G_n$ lying over $e$. 
If any entry $A_{bw}^{ij}$ is zero, we remove that edge from $\G_n$.

Now when expanded over the symmetric group $S_{nN}$, 
nonzero terms in $\det\tilde K$ are in bijection with single-dimer covers $\sigma$
of $\G_n$: a single dimer cover $\sigma$ is a bijection from black vertices $b^j$ to adjacent white vertices $w^i$,
 and has ``weight" $(\tilde K)_\sigma := (-1)^\sigma\prod(A_{bw})_{ij}$ where the product is over dimers in the cover and the sign is the signature of the permutation defined by $\sigma$.

 \begin{figure}[t]
\center{\includegraphics[scale=1.5]{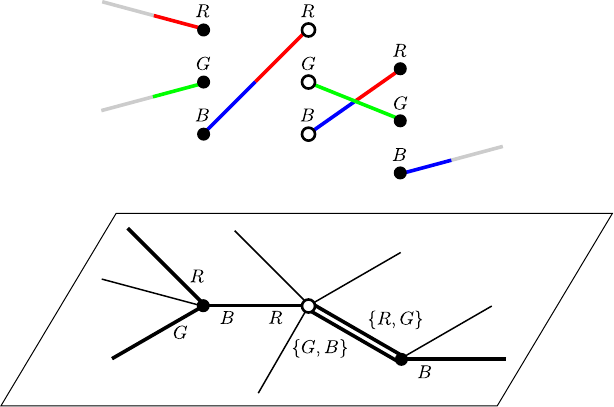}}
\caption{\label{proof}  Example in the case $n=3$.  The graph $\G$ is equipped with a $3$-multiweb $m$ (bottom), and one of many possible lifts of $m$ to $\G_n$ is chosen (top).  Note that only part of $\G$ is shown.  There are two possible lifts of $m$ over this part; the other lift connects G to R, and B to G, over the doubled edge.}
\end{figure}

Each single-dimer cover $\sigma$ of $\G_n$ projects to an $n$-multiweb $m$ of $\G$.
We group all single dimer covers of $\G_n$ according to their corresponding $n$-multiweb $m$.
That is
$$\det\tilde K = \sum_{m\in\Omega_n(\G)}\sum_{\sigma\in m} (\tilde K)_\sigma.$$
We claim that the interior sum is, up to a global sign, $\Tr_L(m)$; this will complete the proof.
That is, we need to prove, for a constant sign $s$ (independent of $m$) and for $m\in\Omega_n(\G)$, 
\be\label{det1}s\Tr_L(m) = \sum_{\sigma\in m}(-1)^\sigma\prod_{e=\tilde b\tilde w}(A_{bw})_{\tilde w\tilde b},\ee
where the product is over edges of the dimer cover $\sigma$ of $\G_n$, or in other words, $\tilde b=\sigma(\tilde w).$
In this product, $\tilde b\tilde w\in E(\G_n)$ lies over an edge $bw$ of the multiweb $m$; the vertex $\tilde w$ corresponds to a choice of color of the half edge at $w$ and $\tilde b$ corresponds to a choice of color of the half edge at $b$. 
The edge $\tilde b\tilde w$ has an associated sign $\eps_{\tilde w\tilde b}=\eps_{wb}$. 

Now we group the $\sigma\in S_{nN}$ according to colorings of the edges of $\G$. 
An edge $e=bw$ of $\G$ of multiplicity $k$ is colored by two sets $S_e,T_e$ of colors, both of size $k$, where $S_e$ is associated to the white vertex and $T_e$ to the black vertex. 
There are then $k!$ corresponding dimer covers of $\G_n$ lying over that edge, one for each bijection $\pi_e$ from $S_e$ to $T_e$.
We group the $\sigma$ into colorings $c$ with the same sets of colors $S_e,T_e$ on each edge. Each permutation $\sigma$
corresponds to a choice of such a coloring $c$ of $m$ and a choice, for each edge $e$ of multiplicity $k$,
of a bijection $\pi_e$ between the sets $S_e$ and $T_e$. 
After this grouping we can write the RHS of (\ref{det1}) as a sum over colorings $c$ of $m$:
\begin{align}\nonumber&= \sum_{c}\sum_{\sigma\in c}(-1)^\sigma\prod_{e=\tilde b\tilde w}\eps_{\tilde w\tilde b}\prod_{e=\tilde b\tilde w}(\phi_{bw})_{\tilde w\tilde b}\\
\nonumber&= \sum_{c}\left(\prod_{e=\tilde b\tilde w}\eps_{\tilde w\tilde b}\right)\sum_{\sigma\in c}(-1)^\sigma\prod_{e=\tilde b\tilde w}(\phi_{bw})_{\tilde w\tilde b}
\end{align}  
where the products are over edges of the dimer cover of $\G_n$, and
where we used the fact that, once the multiplicities $m$ are fixed,  $\prod\eps_{\tilde w\tilde b}$ is independent of $\sigma$ (and, in fact, of $c$ as well). 
(Indeed, $\prod\eps_{\tilde w\tilde b}=\prod_{bw\in E(\G)}(\eps_{wb})^{m_{bw}}$ where $m_{bw}$ is the edge multiplicity.)
Now $\sigma\in S_{nN}$ is a composition $\sigma=(\prod\pi_e)\sigma_0$ of a permutation $\sigma_0$ (which depends only on $c$,
and is the permutation matching each element of each $S_e$ with the corresponding element of $T_e$ when both sets are taken in their natural order) and the individual $\pi_e$. More precisely, we should write $\sigma = (\prod_e \pi'_e)\sigma_0$
where $\pi'_e$ is the bijection from $S_e$ to $S_e$ which, when composed with the natural-order bijection $\sigma_0$ from $S_e$ to $T_e$, gives $\pi_e$.  Thus 
\begin{align}
\nonumber&= \sum_{c}\left(\prod\eps_{\tilde w\tilde b}\right)(-1)^{\sigma_0}\prod_{e=bw}\sum_{\pi'_e}(-1)^{\pi'_e}\prod_{s\in S_e} ((\phi_{bw})_{S_e,T_e})_{s,\pi'_e(s)}\\
&= \sum_{c}\left(\prod\eps_{\tilde w\tilde b}\right)(-1)^{\sigma_0}\prod_{e=bw}\det(\phi_{bw})_{S_e,T_e}.\label{detline}
\end{align}

Recalling (see Proposition \ref{tracecolor0}) the definition of trace of a multiweb we have
\be\label{trline}\Tr_L(m) =\sum_{c}\prod_v c_v\prod_{e=bw}\det(\phi_{bw})_{S_e,T_e}\ee
where the sum is over colorings $c$ of the half-edges, and the product is over edges of $m$.  
There is a one-to-one correspondence between the terms of (\ref{detline}) and those of (\ref{trline}); it remains to compare their signs.

Let us take a \emph{pure} dimer cover of $\G_n$: a dimer covering which matches like colors. Then each $\pi_e'$ is the identity map, and $\sigma=\sigma_0$.
This dimer cover projects to an edge-$n$-coloring of $\G$.  (For each $m$, such an edge-$n$-coloring $c$ exists by Proposition \ref{tracecolorplanar}; furthermore given an edge-$n$-coloring $c$, there is a unique pure dimer cover of $\G_n$ projecting to $c$.)
By Kasteleyn's theorem \cite{Kenyon09Statisticalmechanics}, the sign $(-1)^{\sigma_0} \prod\eps_{\tilde w\tilde b}$ is constant (varying over \textit{all} pure dimer covers of $\G_n$, allowing $m$ to vary as well), 
in fact a constant to the $n$th power,
since each dimer cover of a given color contributes the same sign. 
Likewise by Proposition \ref{tracecolorplanar} and Corollary \ref{poscilia}, $\prod_v c_v$ is a constant
equal to $1$ for either $n$ odd or for $n$ even with positive cilia.

So we just need to compare the sign of an arbitrary coloring with an edge-$n$-coloring. 
Suppose we change a coloring by transposing two colors at a single, without loss of generality white, vertex $w$. 
For the purposes of computing the sign change
we can assume $m$ is 
proper: all multiplicities are $1$, by splitting multiple edges into parallel edges. 
Now in this case under a transposition of colors at $w$ both $\sigma_0$ and $c_w$ change sign, so both
(\ref{detline}) and (\ref{trline}) change sign.
Since (clearly) transpositions at vertices connect the set of all colorings   we see that (\ref{detline}) and (\ref{trline})
agree up to a global sign for all colorings (and all $m$) for $n$ odd, and are actually equal for $n$ even.
\end{proof}

As an application of Theorem \ref{main}, we can put real edge weights $\{x_e\}_{e\in E}$ 
on the edges of $\G$, multiplying the present connection $\Phi$ by these weights to get a new 
$M_n(\mathbb{R})$-connection $\Phi'$. 
Then $\Tr_L(m,\Phi') = \left(\prod_{e\in m}x_e^{m_e}\right)\Tr_L(m,\Phi)$  by multilinearity of the trace.  This gives
$$
\sum_{m\in\Omega_n(\G)}\left(\prod_{e\in m}x_e^{m_e}\right)\Tr_L(m,\Phi) = \pm\det\tilde K(\Phi').$$

This allows us to compute $\Tr_L(m,\Phi)$ in practice for any particular multiweb $m$ as follows.
Put variable edge weights $x_e$ on edges of $\G$
and extract the coefficient of $\prod_{e\in m}x_e^{m_e}$ from $\det \tilde K$, which can be thought of as a polynomial in the formal variables $x_e$ :
\be\label{xyzthm}\Tr_L(m,\Phi) = \pm \left[\prod_{e\in m}x_e^{m_e}\right]\det \tilde K(\Phi').\ee

In practice, we can take $\G$ to be only the part of the graph that underlies the multiweb $m$ whose trace we are interested in computing.

\section{Reductions for \texorpdfstring{$\SLth$}{SL3}}\label{reductions}
\label{sec:sl3skein}

The trace of a large multiweb is mysterious and hard to compute; the method outlined at the end of the previous section is 
only an exponential-time (in the size of the multiweb) algorithm. While we cannot improve generally on the exponential nature of this
computation, we can make it conceptually simpler in the case of $\mathrm{SL}_3$ by applying  certain reductions, or \textit{skein relations}, as described in \cite{Jaeger92DiscreteMath, Kuperberg96CommMathPhys},  
which simplify its computation.  (Skein relations for webs also exist for $\mathrm{SL}_n$ for $n>3$--see e.g. \cite{SikoraTrans01}--but it is 
more challenging to use them to ``reduce" a web, as in Section \ref{reducedwebs} below.)

In light of Proposition \ref{click} for $n=3$, throughout this section we will systematically omit mentioning cilia data.  Also, ``multiweb'' will always mean ``$3$-multiweb''.

\subsection{Webs in surfaces}
\label{ssec:websinsurfaces}

Let $\G$ be a bipartite graph embedded on a surface $\Sigma$. We say $\G$ \emph{fills $\Sigma$}
if every complementary component of $\G$ is a topological disk or disk with a single hole. When $\G$ fills $\Sigma$ 
we call the complementary components \emph{faces of $\G$}.

If $m$ is a multiweb, call a vertex $v$ \emph{trivalent} if $m_e\in\{0,1\}$ for all edges $e$ of $\G$ incident to $v$.  Note a non-trivalent vertex $v$ is  adjacent to an edge $e$ with $m_e=2$ or $m_e=3$. Edges of $m_e=2$ are part of chains
of edges of alternating multiplicity $1$ and $2$ beginning and ending on trivalent vertices (a single edge of multiplicity $1$ between two trivalent vertices is also considered a chain).  Consider the bipartite trivalent graph $\bar m$ with vertex set the set of trivalent vertices of $m$, whose edges correspond to the chains connecting corresponding trivalent vertices in $m$; in particular, tripled edges are forgotten, and closed chains become loops without vertices.  If $\G$ is equipped with an $\SLth$-connection $\Phi$, this determines an $\SLth$-connection on $\bar m$ (by using $\SLth$-gauge transformations to push the part of $\Phi$ on each chain to the black trivalent vertex  end, say), which by abuse of notation we also call $\Phi$.  The web-trace $\Tr(\bar m,\Phi)$ is thus defined and equal to $\Tr(m,\Phi)$.

Given $\G$, we say that two trivalent graphs $\bar m, \bar m'$ constructed from multiwebs $m$ and $m'$ 
are \emph{equivalent} if they are isotopic on the surface.  Let $[m]$ denote the corresponding equivalence class, called a \emph{web}.

Suppose $\G$ fills $\Sigma$ and, moreover, $\Phi$ is a flat $\SLth$-connection, that is, the monodromy around each contractible face is trivial. Then if $\bar m$, $\bar m'$ are isotopic, we have $\Tr(\bar m,\Phi)=\Tr(\bar m',\Phi)$ and so it makes sense to talk about the web-trace $\Tr([m],\Phi):=\Tr(\bar m,\Phi)$ 
of the web $[m]$ on the surface.  

Below, we often suppress the brackets in the expression $[m]$, and simply refer to $m$ as a web.

\subsubsection{Skein relations for webs}
\label{ssec:skeinrelations}

Let $m_1$, $m_2$, $m_3$ be webs as shown in Figure \ref{skeinsquare}, labeled from left to right.  The webs disagree inside a contractible disk as in the figure, and agree outside of this disk.   We assume that the webs are equipped with flat $\SLth$-connections $\Phi_1$, $\Phi_2$, $\Phi_3$, agreeing outside the disk, and such that the connection is the identity on the edges in the disk.  (Note that when the connection is nontrivial, we can locally trivialize the connection on these edges, since the connection is flat.)  The \emph{type-III skein relation} is the following  relation among web-traces:
\begin{equation*}
	\Tr(m_1, \Phi_1)=\Tr(m_2, \Phi_2)+\Tr(m_3, \Phi_3).
\end{equation*}
Similarly, we have the \emph{type-II} and \emph{type-I skein relations}, also taking place in contractible disks, as displayed in Figures \ref{loopremoval} and \ref{loopremoval3}.

\begin{figure}[htb]
\center{\includegraphics[width=4in]{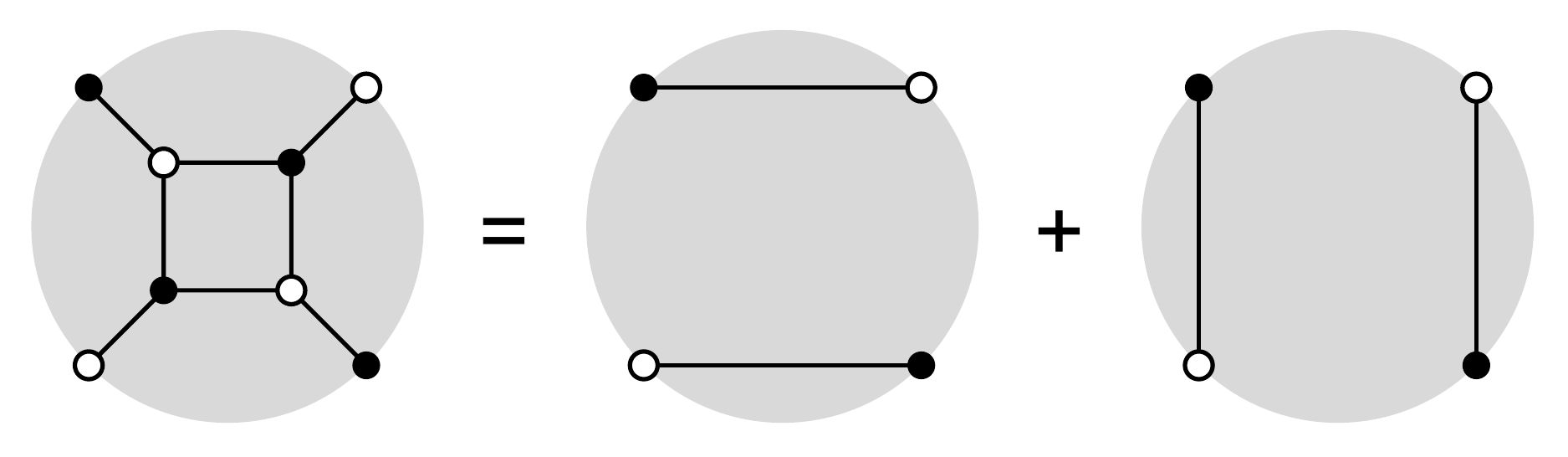}}
\caption{\label{skeinsquare}Type-III:  square removal.}
\end{figure}
\begin{figure}[htb]
\center{\includegraphics[width=2.5in]{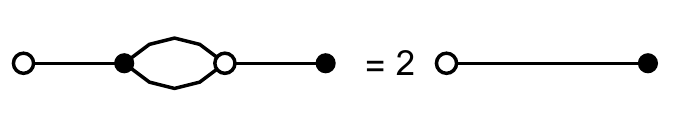}}
\caption{\label{loopremoval} Type-II:  bigon removal.}
\end{figure}
\begin{figure}[htb]
\center{\includegraphics[width=.8in]{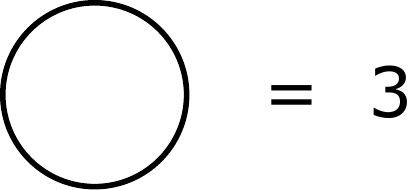}}
\caption{\label{loopremoval3} Type-I:  loop removal.}
\end{figure}

The type-I relation is immediate from the definition of the web-trace.  
We can prove the type-III and type-II relations by using the tensor definition of trace, interpreting terms 
via edge colorings.
We need to check that the signed number of edge colorings of the left- and right-hand sides are the same, for all
possible sets of colors of the external edges.  See Figure \ref{skeincolor} for a proof.

\begin{figure}
\center{\includegraphics[width=4in]{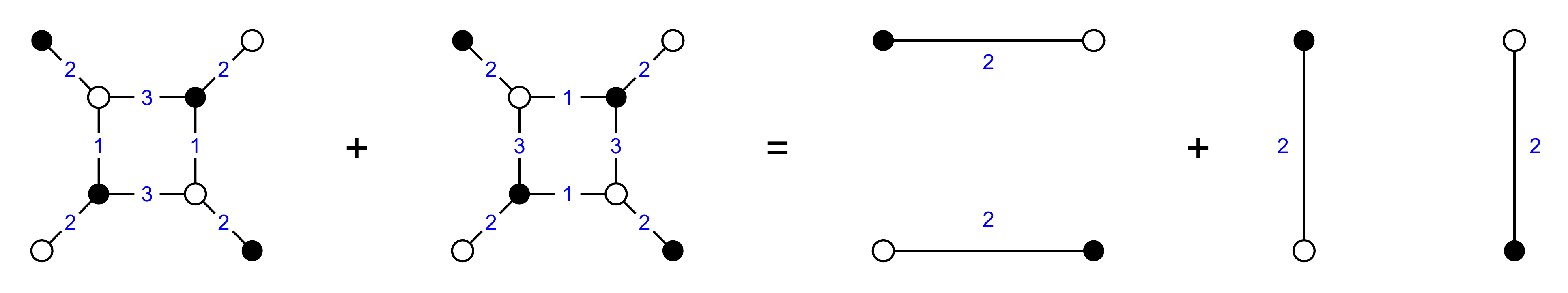}}
\center{\includegraphics[width=2in]{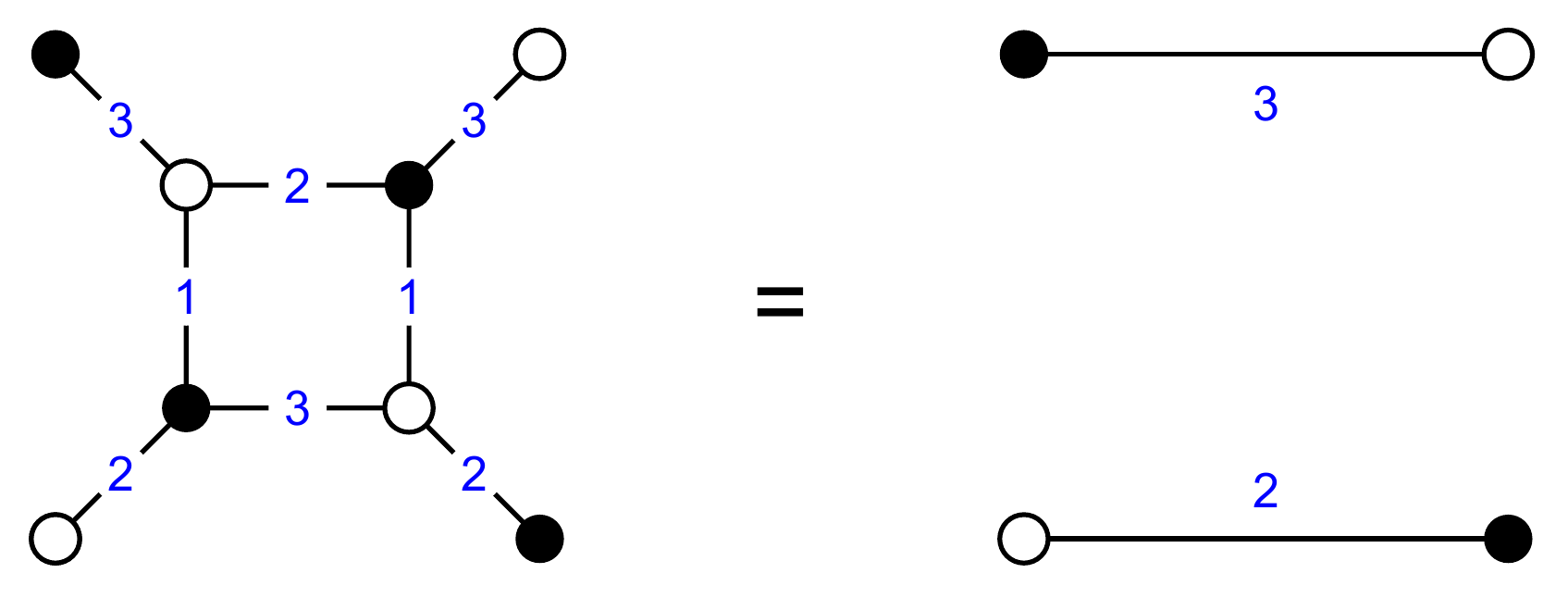}}
\center{\includegraphics[width=4in]{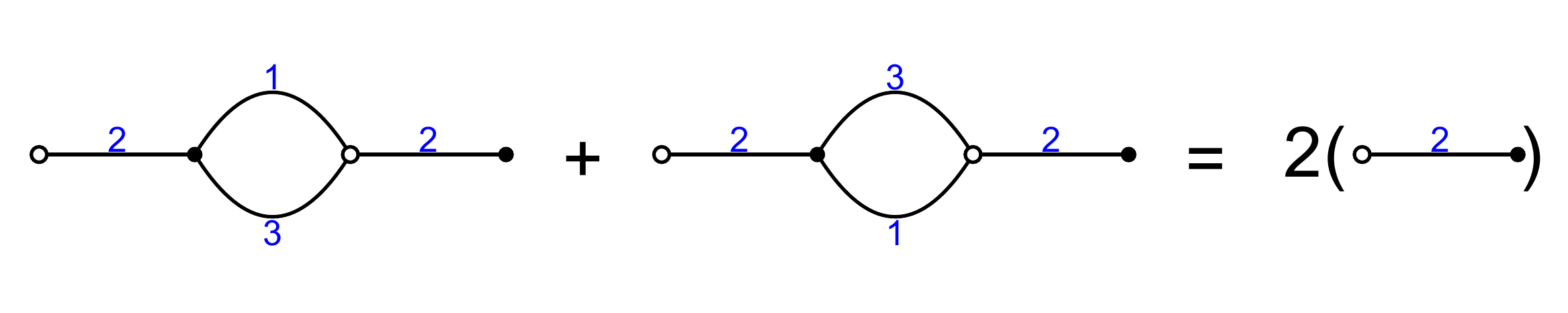}}
\caption{\label{skeincolor}The skein relations for edge colorings: if all four diagonal edges of a square are the same color (color $2$ in this illustration),
there are two ways to color the edges of the central square face with the two remaining colors, and likewise two possible
reductions on the right-hand side. If the diagonal edges have two colors, the like colors must be adjacent, as shown;
then there is only one way to complete the coloring on the edges of the central square face, and one reduction on the right-hand side.
A $2$-gon can be colored in two ways, and can be replaced with a single edge (joining its neighboring edges into a single edge as shown) with a factor $2$.}
\end{figure}

\subsubsection{Skein relations for multiwebs}\label{mwskein}

We can also define a notion of skein relation for multiwebs in  $\G$. 
A \emph{multiweb skein relation} is an operation that takes a multiweb $m\in\Omega_3(\G)$ to a formal linear combination of other multiwebs in $\G$.  (As above these relations take place locally within contractible disks in the underlying surface $\Sigma$).

\begin{figure}
\center{\includegraphics[width=4.2in]{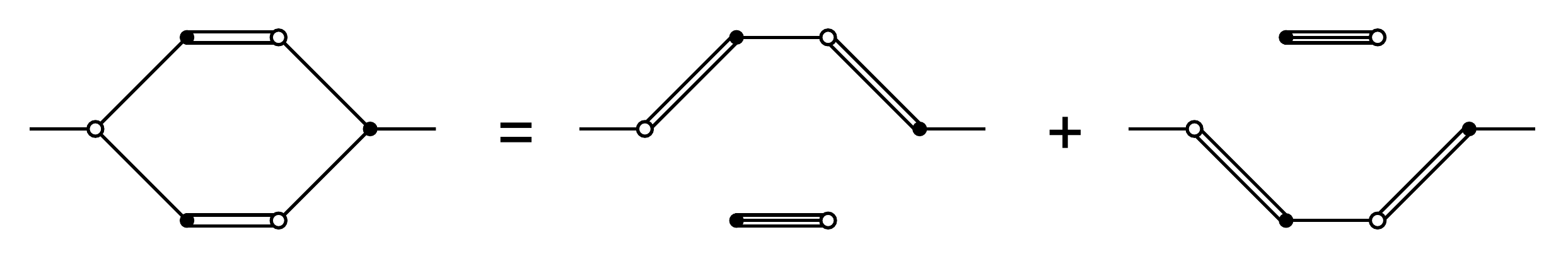}}
\caption{\label{multiwebskein2} Removing a bigon, and replacing with a path in two different ways.}
\end{figure}

\begin{figure}
\center{\includegraphics[width=4.2in]{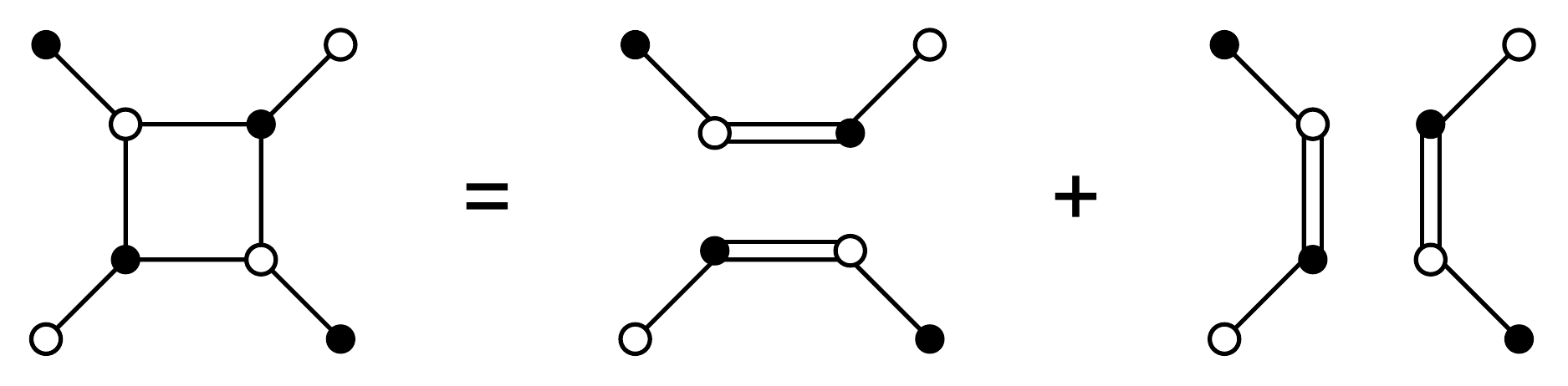}}
\caption{\label{multiwebskein1} Removing a  square.}
\end{figure}


The type-I, loop-removal, skein relation has a multiweb version which works as follows. Take a 
closed chain (consisting in an alternating sequence of single and double edges); this loop
may surround other vertices of $\G$, but the loop must be contractible in $\Sigma$. Replace the loop with a sequence of tripled edges, by increasing the multiplicity of the doubled edges and decreasing the multiplicity of the single edges along the loop; then multiply the resulting multiweb by $3$. 

The skein relations of type-II and type-III above also
have multiweb versions as shown in Figure \ref{multiwebskein2},\ref{multiwebskein1}. For the type-II, we take two vertices which have two disjoint paths
joining them as shown, not necessarily of the same length.
We replace this ``bigon" (which must be contractible in $\Sigma$) by the two webs shown:  each is obtained by increasing
the multiplicity on every other edge by $1$ and decreasing the multiplicity on the remaining edges.
For type-III, we have four vertices $a,b,c,d$ connected into a cycle by four paths (which may have lengths larger than $1$, unlike as shown). On condition that this cycle is contractible in $\Sigma$, this cycle is replaced by two multiwebs, each again obtained by increasing the multiplicity of every other edge of the cycle by $1$ and decreasing the multiplicity of the other edges of the cycle.

\subsubsection{Basic skein relation for immersed webs}
\label{sssec:basicskein}

Although we will not use this in this paper, there is a more basic skein relation, for \emph{immersed} graphs.
Suppose that $\G$ is immersed in $\Sigma$, that is, drawn on $\Sigma$ allowing edge crossings 
but without crossings at vertices. Web traces are still defined for immersed graphs with choices of cilia,
since the web trace only relies on a linear order at vertices.
The basic skein relation for web-traces among immersed webs is shown in Figure \ref{skeinsmall}. Note that 
we use the signs 
for traces as in Sikora \cite{SikoraTrans01}, using opposite circular orders at black and white vertices (see Section \ref{ssec:graphsinsurfaces}), which differ from those of Kuperberg \cite{Kuperberg96CommMathPhys} (with Sikora's sign convention we get sign-coherence in Theorem \ref{main}).  The type-II and type-III skein relations (Figures \ref{loopremoval} and \ref{skeinsquare}) can be derived from the basic skein relation and the type-I skein relation (Figure \ref{loopremoval3}), see e.g. \cite{Jaeger92DiscreteMath} or by inspection.  The basic skein relation can also be proved by an edge-coloring argument
which we leave to the reader.  
\begin{figure}[htbp]
\center{\includegraphics[width=3.7in]{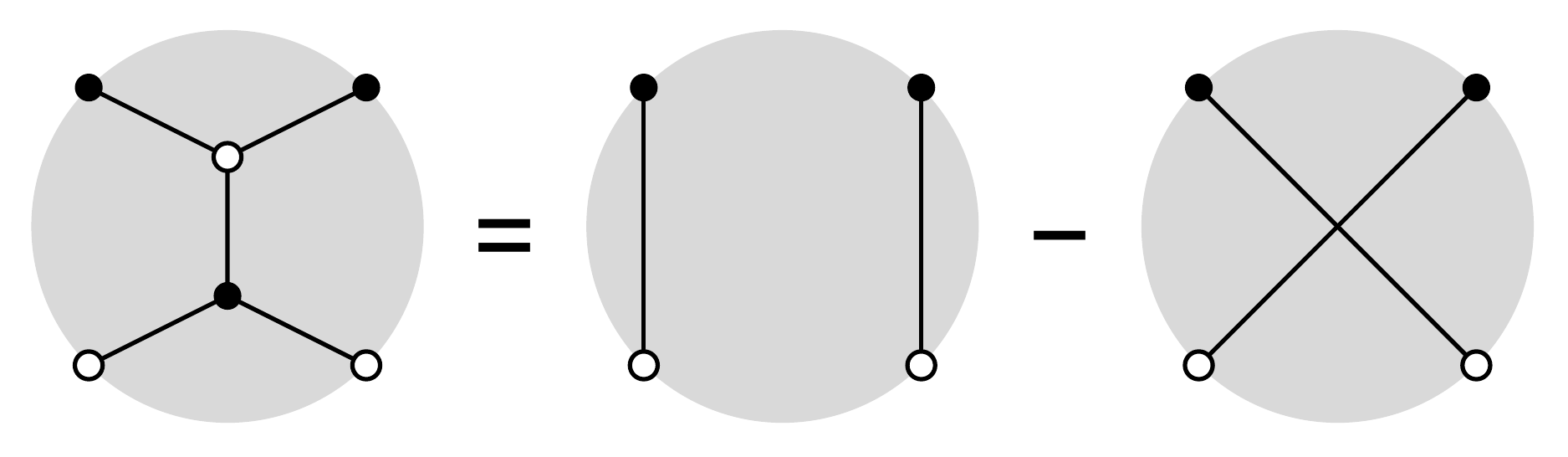}}
\caption{\label{skeinsmall}Basic skein relation for $\SLth$.
}
\end{figure}

\subsection{Planar webs}\label{reducedwebs}

We are ready to combine elements of both Sections \ref{ssec:skeinrelations} and \ref{mwskein}.  We 
suppose $\G$ is embedded in and fills a surface $\Sigma$.

We say that the multiweb $m\in\Omega_3(\G)$ is \emph{reduced} (or \emph{nonelliptic}), and its corresponding web $[m]$ is  
reduced, if $[m]$ has no contractible loops, bigons or quadrilateral faces. Thus in a reduced web all contractible faces have $6$ or more edges.  

Note that $m$ being reduced is a purely topological property, without reference to any connection on the graph.

Given an unreduced multiweb $m$ in $\G$, we can apply multiweb skein relations to $m$ to reduce it to a formal linear combination of reduced multiwebs: $m \mapsto \sum c' m'$.  Such a sequence of reductions is not unique, due to the type-III multiweb skein relation; indeed, it is easy to produce examples of $m$ having more than one quadrilateral face, where different choices of sequences of reductions of the quadrilateral faces give different end ``states'' $\sum c' m'$.  
 
Denote by $\Lambda_3(\G)$ the set of  
reduced webs $[m']$  varying over reduced multiwebs $m'$ in $\G$.  For technical reasons, we need to extend here the notion of equivalence to include an \emph{annulus move} relation identifying reduced webs differing by  oriented boundary loops of an embedded annulus, as in Figure \ref{annulusswap}; this does not affect any of the statements of Section \ref{ssec:websinsurfaces}.  (Essentially, this extra ``isotopy'' relation is an artifact of our working in the 2-dimensional setting of the surface $\Sigma$, rather than the 3-dimensional setting of the thickened surface $\Sigma \times (0, 1)$; compare \cite[Section~5]{SikoraAlgGeomTop07}.)

It was shown in \cite{SikoraAlgGeomTop07, Kuperberg96CommMathPhys} that, although the  end states 
$\sum c^\prime m^\prime$ of reduced multiwebs depend on the sequence of reductions, the formal linear combination $\sum c' [m']$ of the corresponding reduced webs in $\Lambda_3(\G)$ does not.  Also note that if the entire multiweb $m$ lies in a contractible region on $\Sigma$ (for example if $m$ consists only of tripled edges and contractible closed chains), then its reduction will be a positive integer times the class $[\emptyset] \in \Lambda_3(\G)$ of the ``empty'' web.  

\begin{figure}[htb]
\center{\includegraphics[width=2in]{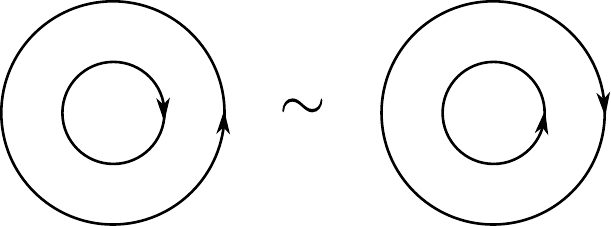}}
\caption{\label{annulusswap} Swapping the boundary loops of an embedded annulus.}
\end{figure}

At the level of traces, let us assume $\G$ is equipped with a flat $\SLth$-connection $\Phi$.  Then $$\Tr(m) = \sum_{\lambda\in\Lambda_3(\G)} C_{m,\lambda}\Tr(\lambda)$$
where $C_{m,\lambda}$ is the number of reduced webs $\lambda \in\Lambda_3(\G)$ ``contained" in $m$, that is, resulting from the reduction of $m$.

Now, for $\Sigma$ a planar surface (Section \ref{planarwebscn}), Theorem \ref{main} writes the determinant $\det\tilde K$  as a sum over unreduced multiwebs $m \in\Omega_3(\G)$. We can further reduce each multiweb to write $\det\tilde K$ as
a sum of traces of reduced webs:

\begin{thm} \label{mainred}Let $\G$ be a bipartite graph $\G$ embedded in and filling a planar surface $\Sigma$, namely
a genus zero surface minus $k\ge0$ disjoint closed disks.
Suppose $\G$ is equipped with a flat
$\SLth(V)$-connection $\Phi$. Then for any choice of Kasteleyn matrix $K$ we have
$$\pm \det\tilde K(\Phi) = \sum_{\lambda\in\Lambda_3(\G)}C_{\lambda}\Tr(\lambda)$$ where the sum is over reduced webs $\lambda\in\Lambda_3(\G)$,
and the $C_\lambda$ are positive integers.
\end{thm}  

The most interesting cases are for $k \geq 2$.  Indeed, when $k=0,1$ then $\Lambda_3(\G)$ consists only of the class of the empty web, $[\emptyset]$.

We describe how to extract the coefficients $C_\lambda$ in some simple cases, for $k=2$ and $k=3$,  in the next section.

Due to the dependence on choices when reducing multiwebs,  we do not have at present a canonical probability measure on the subset of $\Omega_3(\G)$ of reduced multiwebs.  However, reduction does induce a canonical probability measure $\nu_3$ on the set of  reduced webs, $\Lambda_3(\G)$.  
It is defined by (recalling~\eqref{Zndtrivial})
\be\label{nu3def}\Pr(\lambda) = \frac{C_{\lambda} \Tr(\lambda, I)}{Z_{3d}}\ee
where $\Tr(\lambda, I)$ is the web-trace of $\lambda\in\Lambda_3(\G)$ for the identity connection.

This is in contrast to the case of $\mathrm{SL}_2$, where $\nu_2$ is defined on the subset of $\Omega_2(\G)$ of reduced $2$-multiwebs, namely chains of single edges, and doubled edges.  Here, the only $2$-multiweb  skein relation resolves a contractible chain, in two different ways, into a sequence of doubled edges.

\section{Multiwebs on simple surfaces}

We continue studying the case $n=3$.
Sikora and Westbury \cite[Theorem 9.5]{SikoraAlgGeomTop07} showed that reduced webs form a basis for the ``$\mathrm{SL}_3$-skein algebra" for any surface. That is, using skein relations any web can be reduced to a unique linear combination of non-elliptic webs.
Thus \cite[Theorem 3.7]{SikoraTrans01}, the traces for nonelliptic webs form a basis for the algebra of invariant regular functions on the space of flat  $\mathrm{SL}_3$-connections modulo gauge
(the ``$\mathrm{SL}_3$-character variety").

\subsection{Annulus}\label{annulussection}

We consider here the case 
where the graph $\G$ is embedded on an annulus.
The result of \cite{SikoraAlgGeomTop07} for the annulus can be stated as follows.

\begin{prop}\label{planarred}
By use of skein relations of types I, II, and III, any web on an annulus with a flat $\SLth$-local system
can be reduced to a unique positive integer linear combination
of collections of disjoint noncontractible oriented cycles.
\end{prop}  

\begin{proof} For a bipartite connected trivalent graph on a sphere
we have (by Euler characteristic) 
\be\label{EC}6=2n_2+n_4+0n_6 - n_8+\dots+(3-k)n_{2k}+\dots\ee
where $n_k$ is the number of faces of degree $k$. When $\G$ is embedded on an annulus (with no boundary points)
at most two faces contain boundary components, so
there is at least one contractible degree-$2$ face or $2$ contractible degree-$4$ faces.
We first perform type-$2$ skein relations to remove any contractible degree-$2$ faces.
Then there remain at least $2$ faces of degree $4$, upon which we can perform a type-$3$ skein reduction.
This reduces the number of vertices in $\G$ and possibly disconnects $\G$; 
continue with each component until each component is a loop.
\end{proof}

Let $W_{j,k}$ be the set of isotopy classes (allowing loop-swapping isotopies as in Figure \ref{annulusswap})
of reduced webs on $\Sigma$ with $j$ loops of homology class $+1$ (counterclockwise)  and $k$ loops of homology class $-1$ (clockwise). (We orient a $3$-multiweb which is a loop so that the single edges are oriented from black to white,
and the doubled edges are oriented from white to black. See Section \ref{loop}.)

Suppose $\Phi$ is a flat connection with monodromy
$A\in\SLth$ around the generator of $\pi_1$. 
A noncontractible simple loop has trace $\tr(A)$ or $\tr(A^{-1})$ depending on orientation. 
For a general web $m$, by Proposition \ref{planarred}, $\Tr(m)$ will be a polynomial in $\tr(A)$ and $\tr(A^{-1})$ whose coefficients are nonnegative integers: 
$$\Tr(m) = \sum_{j,k\ge0}M_{j,k}\tr(A)^j\tr(A^{-1})^k$$
where $M_{j,k}$ counts reduced subwebs in $W_{j,k}$.

\begin{figure}
\center{\includegraphics[width=2.5in]{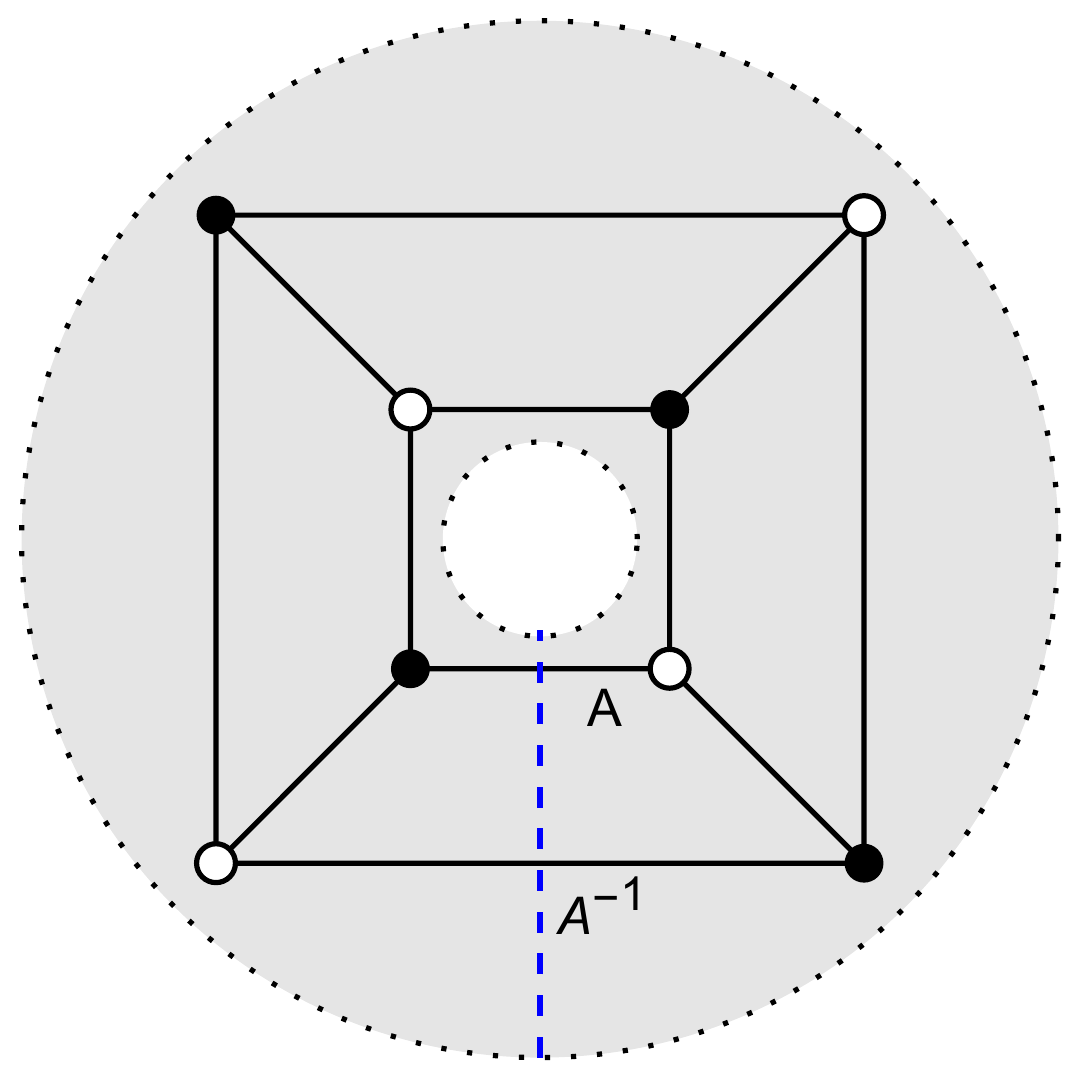}}
\caption{\label{sqweb} A web on an annulus with flat connection having monodromy $A$. We can for example
put parallel transports $A, A^{-1}$ as marked and $I$ on the remaining edges.}
\end{figure}
For example for the web of Figure \ref{sqweb}
the  trace is 
$$\Tr(m) = 15+\tr(A)\tr(A^{-1}).$$

By Theorems \ref{main} and \ref{mainred}, for a graph $\G$ on an annulus we have  
$$\pm \det \tilde K(A) =\sum_{m\in\Omega_3(\G)}\Tr(m) = \sum_{\lambda\in\Lambda_3(\G)}C_\lambda \Tr(\lambda) = \sum_{j,k\ge0} C_{j,k}\tr(A)^j\tr(A^{-1})^k$$
where $C_{j,k}$ is the total contribution of all the $M_{j,k}$'s.  

We can compute $C_{j,k}$ concretely as follows. Assume $A$ has eigenvalues $x,y,z$ (with $xyz=1$).  
Let $u=\frac13Tr(A),v=\frac13Tr(A^{-1})$. Then $x,y,z$ are roots of $p(\lambda) = \lambda^3-3u\lambda^2+3v\lambda-1$. 
We can assume without loss of generality that $A$ is a diagonal matrix with eigenvalues $x,y,z$. Then if we rearrange rows and columns of $\tilde K$ so that rows and columns of index $1\bmod 3$ come first, then $2\bmod 3$ then $0\bmod 3$, the new matrix is of block form 
$$S^{-1}\tilde KS = \begin{pmatrix}K(x)\\&K(y)\\&&K(z)\end{pmatrix}$$
where the $K(\cdot)$ are the corresponding scalar matrices.
Thus 
\be\label{tripleK}\det\tilde K = \det K(x) \det K(y)\det K(z)\ee
This is a symmetric polynomial of $x,y,z$ and so can be written as a polynomial in $u,v$. 

As a concrete example,
take a $2m\times n$ square grid $\G_{2m,n}$ on an annulus (with circumference $2m$, with $m$ odd, and height $n$), see Figure \ref{gridannulus}.
By Proposition \ref{anndet} in the appendix, for $n$ even, 
$$\det K_{2m,n}(x) = \pm \prod_{j=1}^{n/2}\frac{(x+\alpha_j^{2m})(x+\alpha_j^{-2m})}{x}$$
where $\alpha_j=-\cos\theta_j+\sqrt{1+\cos^2\theta_j}$ and $\theta_j=\frac{\pi j}{n+1}$.
Note $(x+r)(y+r)(z+r)= 1+3vr+3ur^2+r^3.$ Hence using (\ref{tripleK}),
\be\label{Kprod1}\det\tilde K = \pm\prod_{j=1}^{n/2}(1+3v\alpha_j^{2m}+3u\alpha_j^{4m}+\alpha_j^{6m})(1+3v\alpha_j^{-2m}+3u\alpha_j^{-4m}+\alpha_j^{-6m}).\ee

Now we have
$$P(u,v) :=\frac{\det\tilde K(u,v)}{\det\tilde K(1,1)}=\sum_{p,q\geq 0} c_{p,q}u^pv^q,$$ 
where $c_{p,q}$ is the probability, for $\nu_3$, of a reduced web of type $p,q$ (see (\ref{nu3def})).
Thus $P(u,v)$ is the probability generating function for $\nu_3$. 
Each factor $1+Au+Bv+C$ in (\ref{Kprod1}), which is a polynomial in $u,v$ with nonnegative coefficients, can be scaled so that its coefficients add to $1$.
It is then the probability generating function for a distribution on $\Z^2$ with support $\{(0,0),(0,1),(1,0)\}$.  Thus $P(u,v)$ 
itself can be interpreted as the probability generating function for an $n$-step random walk in $\Z^2$ starting from $(0,0)$, 
$$(X,Y)=(X_1,Y_1)+(X_{-1},Y_{-1}) + (X_2,Y_2) +\dots+(X_{-n/2},Y_{-n/2})$$
where for $j>0$ the step $(X_j,Y_j)$ is $(0,0),(0,1),(1,0)$ with probabilities respectively
$$\frac{1+\alpha_j^{6m}}{1+3\alpha_j^{2m}+3\alpha_j^{4m}+\alpha_j^{6m}}, \frac{3\alpha_j^{2m}}{1+3\alpha_j^{2m}+3\alpha_j^{4m}+\alpha_j^{6m}}, \frac{3\alpha_j^{4m}}{1+3\alpha_j^{2m}+3\alpha_j^{4m}+\alpha_j^{6m}}$$
and for $j<0$,  the step $(X_j,Y_j)$ is $(0,0),(0,1),(1,0)$ with probabilities respectively
$$\frac{1+\alpha_j^{6m}}{1+3\alpha_j^{2m}+3\alpha_j^{4m}+\alpha_j^{6m}}, \frac{3\alpha_j^{4m}}{1+3\alpha_j^{2m}+3\alpha_j^{4m}+\alpha_j^{6m}}, \frac{3\alpha_j^{2m}}{1+3\alpha_j^{2m}+3\alpha_j^{4m}+\alpha_j^{6m}}.$$

We have the mean values
$$\bar X=\bar Y= \sum_{j=1}^{n/2} \frac{3\alpha_j^{2m}+3\alpha_j^{4m}}{1+3\alpha_j^{2m}+3\alpha_j^{4m}+\alpha_j^{6m}}= \sum_{j=1}^{n/2} \frac{3\alpha_j^{2m}}{(1+\alpha_j^{2m})^2}.$$
These expected values can be interpreted as the expected value of the number of nontrivial loops of each orientation on the annulus.

Now let $m,n$ tend to $\infty$ with $m/n\to\tau$. 
Let us estimate $\bar X,\bar Y$ for $m,n$ large. We only get a nontrivial contribution if $|\alpha_j|\approx 1$, that is, when $j\approx n/2$.   
We can write (letting $\ell=n/2-j$)
$$\alpha_\ell= 1-\frac{\pi(\ell+\frac12)}{n} + O(\frac{\ell^2}{n^2})$$
and so
$$\alpha_\ell^{2m} = q^{2\ell+1}(1+o(1))$$
where $q=e^{-\pi\tau}$.
We thus have, up to errors tending to zero as $m,n\to\infty$,
\be\label{mean}\bar X = \bar Y = \sum_{\ell=0}^{\infty}\frac{3q^{2\ell+1}}{(1+q^{2\ell+1})^2}.\ee

Now suppose $\tau$ is large: we have a long thin annulus. Then $q$ is small.
From (\ref{mean}), since $X,Y$ are nonnegative $\Z$-valued, 
we have $$\frac{\det\tilde K(u,v)}{\det\tilde K(1,1)}=1+o_q(1).$$
Thus the probability of having zero crossings tends to $1$ as $q\to0$. 

The probability $c_{j,k}$ of a $(j,k)$-crossing (a reduced 
web containing exactly $j$ and $k$ nontrival loops in each orientation after reduction) for $j,k\ge0$
is to leading order $q^{\delta_{j,k}}3^{j+k}$ where the ``crossing exponent" is
$\delta_{j,k} = \lceil \frac{2(j^2+jk+k^2)}3\rceil$. This can be seen by expanding (\ref{Kprod1}) and extracting the appropriate term to leading order
(see the calculation in the appendix, Section \ref{extract}.)

A similar computation can be done for an $n$-multiweb on an annulus, for $n>3$, since \cite{DanandTom} shows that an $n$-multiweb on an annulus can be reduced to a collection of loops (of $n-1$ types: with multiplicities $(k,n-k)$ for $k=1,\dots,n-1$).

\subsection{Pair of pants}\label{pantssection}

\begin{figure}
\begin{center}
\includegraphics[width=4cm]{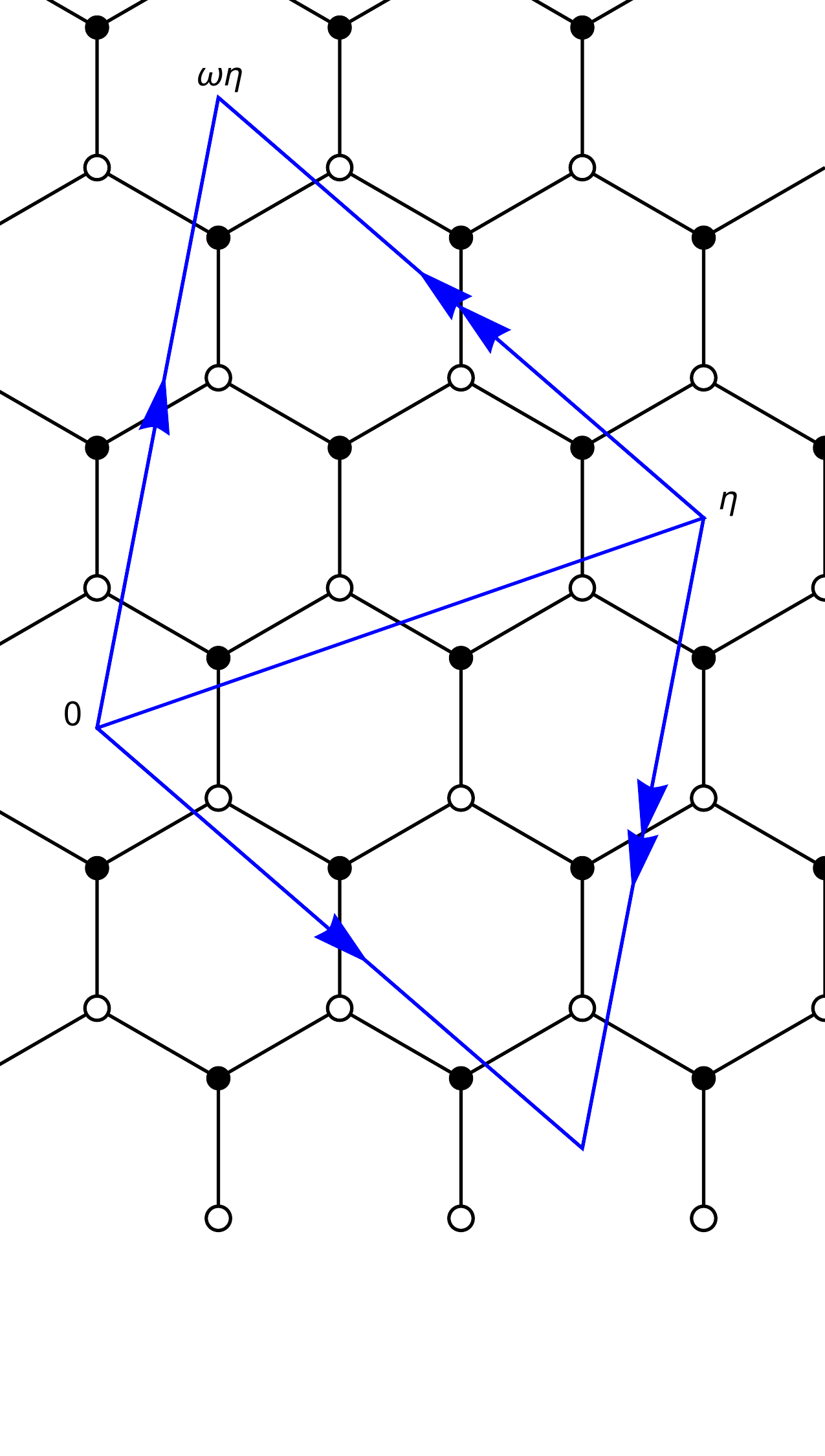}
\end{center}
\caption{\label{pantswebs} Non-elliptic webs on a pants. Shown is $W_\eta$ for $\eta=2+\omega$.}
\end{figure}

Let $\Sigma$ be a pair of pants, that is, a sphere with three holes.
Let $\omega=e^{\pi i/3}$ and let $\eta\in\C$ have the form $\eta=a+b\omega$ where $a,b\in\Z_{\ge0}$.
We construct a reduced web $W_{\eta}$ on $\Sigma$ as follows; see Figure \ref{pantswebs}.
We take the rhombus with sides $\eta\omega^{-1}$ and $\eta\omega$ in $\C$, and glue sides as shown, putting punctures at the corners, to form a $3$-punctured sphere $\Sigma$.
The image
of the standard honeycomb graph (dual to the regular triangulation of $\C$) descends to a reduced web on $\Sigma$.

\begin{prop}
A reduced web on a pair of pants is a union of a collection of loops and at most one of the webs $W_\eta$. 
\end{prop}
\begin{proof}
If the web has a bigon or quad face not containing a boundary component, it is not reduced.  If any boundary face is not
a bigon, there must be (by (\ref{EC})) a non-boundary quad face, so it is not reduced. So if reduced, all boundary faces are bigons, and all other
faces are hexagons, again by (\ref{EC}). 
The dual of such a web is a triangulation with all vertices of degree $6$ except for three vertices of degree $2$. Geometrically (replacing triangles with equilateral triangles) it is a $(3,3,3)$-orbifold. Such a space has a $3$-fold branched cover (over the vertices of degree $2$) which is an equilateral torus. It is a quotient of the regular hexagonal
triangulation by a hexagonal sublattice. So these triangulations are indexed by Eisenstein integers $\eta=a+be^{\pi i/3}$,
where $a,b\in\Z_{\ge0}$,
and $a^2+ab+b^2$ is the number of white vertices (or black vertices).
\end{proof}

Note that there are two possible orientations of each $W_\eta$ (obtained from switching the colors), the other denoted $W_\eta^*$, but at most one can occur in any reduced web.

Letting $A,B,C=AB^{-1}$ be the monodromies around the punctures of a flat connection on $\Sigma$, we have
$$\det\tilde K = \sum_{i_1,i_2,j_1,j_2,k_1,k_2,\eta}C_{\vec{i}}\tr(A)^{i_1}\tr(A^{-1})^{i_2}\tr(B)^{j_1}\tr(B^{-1})^{j_2}\tr(C)^{k_1}\tr(C^{-1})^{k_2}\Tr(W_\eta)$$
for some integers $C_{\vec{i}}=C_{i_1,i_2,j_1,j_2,k_1,k_2,\eta}$. 

While extracting the coefficients $C_{\vec{i}}$ can be done in principle (by the result of Sikora-Westbury mentioned at the beginning of this section), in practice it is not easy.
\medskip

\noindent{{\bf Open Problem:} Extract the coefficients $C_{\vec i}$ in the above expression.}
\medskip

Let us only consider one simplified situation.

Suppose a bipartite graph $\G$ is embedded on $\Sigma$, and two of the three boundary components of the pants are in \emph{adjacent} faces of $\G$. Then subwebs of $\G$ of type $W_\eta$ can only occur if $\eta=1$,
that is, except for loop components a reduced web $W$ in $\G$ can only be a $W_1$ (or a $W_1^*$, depending on the orientation of the edge between the adjacent faces); such a web is homeomorphic to a theta graph, see Section \ref{sssec:thetagraph}.  

For the identity connection we then have 
$\det\tilde K = Z_0+Z_1\Tr(W_1) $, where $Z_0$ is the weighted sum of multiwebs not containing
$W_1$ in their reduction,  and $Z_1$ is the weighted sum of reduced subwebs containing
a component of type $W_1$. 
We can compute $Z_0,Z_1$ as follows.

Suppose we impose a flat connection with monodromy $A,B$ around the generators of $\pi_1$, where $A,B$
are chosen so that traces of simple loops are $3$ and the trace of a $W_1$ is a variable.  
For example $A=\begin{pmatrix}1&a&1\\0&1&1\\0&0&1\end{pmatrix}$ and $B=\begin{pmatrix}1&0&0\\1&1&0\\-a&-a^2&1\end{pmatrix}$ so that $Tr(AB)=3-a^2$.

Note that $Tr(W_1)=Tr(A)Tr(B)-Tr(AB)=6+a^2$ and the traces of $A,B,C$ and their inverses do not involve $a^2$. So $Z_1$ is the coefficient of $a^2$ in $\det \tilde K$.

\section{Appendix: annulus determinant}

\begin{prop}\label{anndet} For the grid graph $\G_{2m,n}$ on an annulus as in Figure {\upshape\ref{gridannulus}} with $n$ even and $m$ odd we have 
$$\det K_{2m,n}(z) =  \pm\prod_{k=1}^{n/2}\frac{(z+\alpha_k^{2m})(z+\alpha_k^{-2m})}{z}$$
where  $\alpha_k=-\cos\theta_k+\sqrt{1+\cos^2\theta_k}$ with $\theta_k=\frac{\pi k}{n+1}$.
If $n$ is odd and $m$ is odd the result is 
$$\det K_{2m,n}(z) =  \pm(\sqrt{z}+\frac1{\sqrt{z}})\prod_{k=1}^{(n-1)/2}\frac{(z+\alpha_k^{2m})(z+\alpha_k^{-2m})}{z}$$
with $\alpha_k$ as above.
\end{prop}

\begin{figure}[t]
\begin{center}\includegraphics[width=3in]{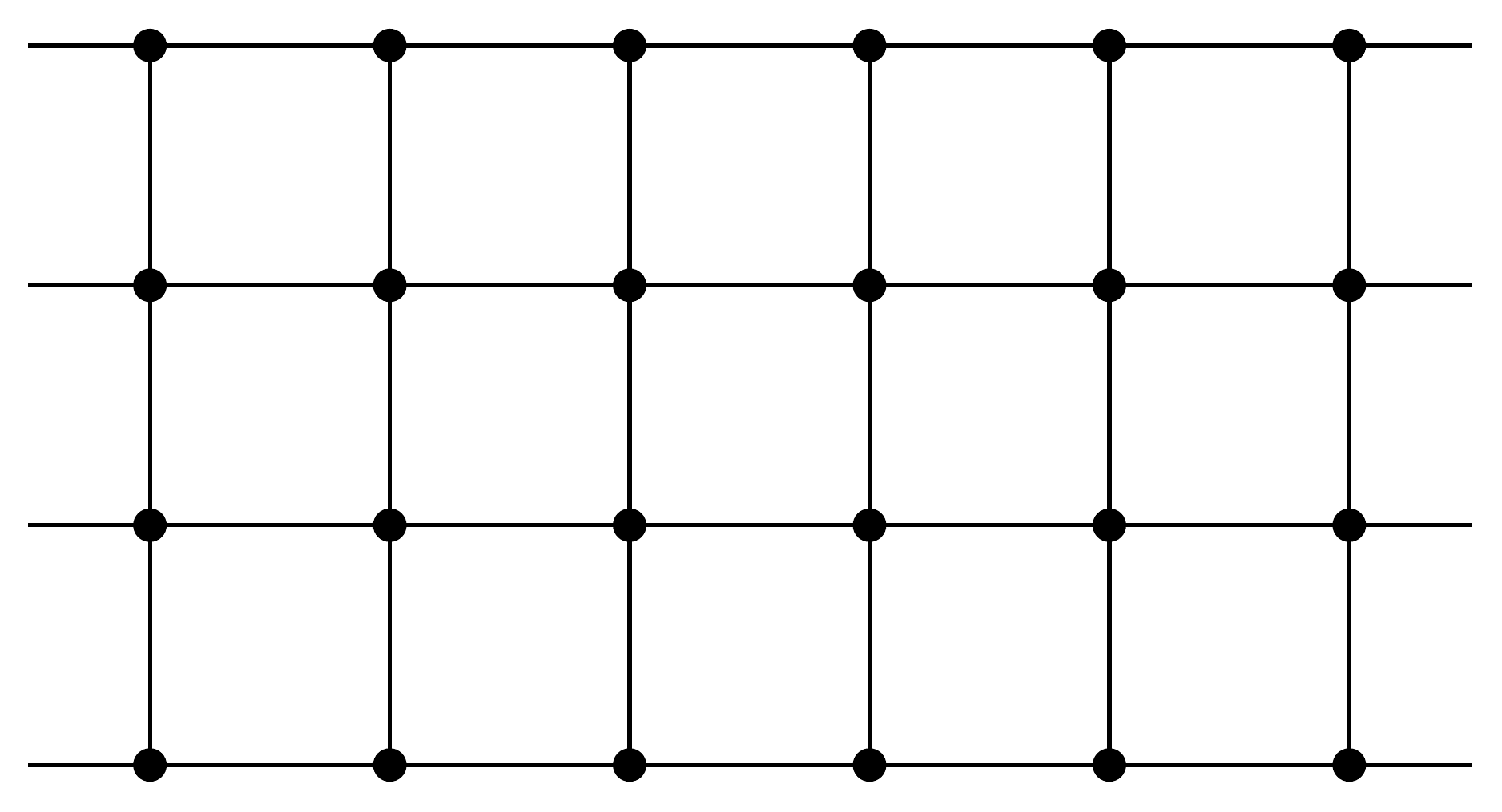}
\end{center}
\caption{\label{gridannulus} $\G_{6,4}$ is obtained by gluing the left and right sides of this figure.
To make a Kasteleyn matrix with flat connection having monodromy $z$ around the annulus, take $\zeta$ such that
$\zeta^{2m}=z$ and put connection $\zeta$ on all horizontal edges, directed east (and $1$ on vertical edges).}\end{figure} 

\begin{proof}
Recall that the Kasteleyn matrix $K$ has rows indexing white vertices and columns indexing black vertices.
We consider here the \emph{large} Kasteleyn matrix $\hat K$, with rows (and columns) indexing \emph{all} vertices, both black and white. 
We have $\hat K(z) = \begin{pmatrix}0&K(z)\\K^t(1/z)&0\end{pmatrix}$ and, by symmetry,
$\det K(z) = \det K(1/z)$ so 
$\det K(z)= \pm \det\hat K(z)^{1/2}.$ (The sign depends on choice of gauge and vertex order.)  

We put Kasteleyn ``signs" $i=\sqrt{-1}$ on vertical edges and $1$ on horizontal edges, as in \cite{Kenyon09Statisticalmechanics}.

Indexing vertices by their $x,y$-coordinates, where $(x,y)\in[0,2m-1]\times[1,n]$, the eigenvectors of $\hat K(z)$  are
$$f_{j,k}(x,y) = e^{2\pi ijx/(2m)}\sin\frac{\pi ky}{n+1}$$ 
where $j\in\{0,..,2m-1\}$ and $k\in\{1,2,\dots,n\}$.    
The corresponding eigenvalues are 
$$\lambda_{j,k} = \zeta e^{2\pi ij/(2m)}+\zeta^{-1}e^{-2\pi ij/(2m)} + 2i\cos\frac{\pi k}{n+1}.$$

Thus
\begin{align*}\det K &= \pm\left[\prod_{j=0}^{2m-1}\prod_{k=1}^{n}\zeta e^{2\pi ij/(2m)}+\zeta^{-1}e^{-2\pi ij/(2m)} + 2i\cos\frac{\pi k}{n+1}\right]^{1/2}\\
&= \pm\left[\prod_{k=1}^{n}\prod_{\zeta^{2m}=z}\zeta +\zeta^{-1} + 2i\cos\frac{\pi k}{n+1}\right]^{1/2}\\
&= \pm\left[\prod_{k=1}^{n}\prod_{\zeta^{2m}=z}\frac{(\zeta-\beta_k)(\zeta-\gamma_k)}{\zeta}\right]^{1/2}\\
&= \pm\left[\prod_{k=1}^{n}\frac{(z-\beta_k^{2m})(z-\gamma_k^{2m})}{-z}\right]^{1/2}
\end{align*}
where $\beta_k,\gamma_k$ are roots of the quadratic $\zeta(\zeta +\zeta^{-1} + 2i\cos\frac{\pi k}{n+1})$, that is,
$\beta_k,\gamma_k = i(-\cos\theta_k\pm\sqrt{1+\cos^2\theta_k})$ with $\theta_k=\frac{\pi k}{n+1}$. 
If $n$ is even
we can pair the $k$ and $n+1-k$ terms which are identical, to get
$$\det K=\pm\prod_{k=1}^{n/2}\frac{(z-\beta_k^{2m})(z-\gamma_k^{2m})}{z}.$$

If $n$ is odd the $k=(n+1)/2$ term is $-(z+1/z+2)$ (since $m$ is odd), yielding
$$\det K=\pm(\sqrt{z}+\frac1{\sqrt{z}})\prod_{k=1}^{(n-1)/2}\frac{(z-\beta_k^{2m})(z-\gamma_k^{2m})}{z}.$$

Letting $\alpha_k=-i\beta_k$ (and noting that $\gamma_k=\beta_k^{-1}$, and $m$ is odd) gives the result.
\end{proof}

\section{Appendix: coefficient extraction}\label{extract}

For the computation at the end of Section \ref{annulussection}, we are interested in the behavior of $P(u,v)=\sum_{j,k\geq0} c_{j,k}u^j v^k=\frac{\det\tilde K(u,v)}{\det\tilde K(1,1)}$
when $q$ is small, see (\ref{Kprod1}).
We compute here the exponent of $q$ in the leading-order term of the coefficient of $u^j v^k$ in the expansion
$$ \prod_{i=0}^\infty \frac{(1+3q^{2i+1}u+3q^{4i+2}v+q^{6i+3})(1+3q^{2i+1}v+3q^{4i+2}u+q^{6i+3})}{(1+3q^{2i+1}+3q^{4i+2}+q^{6i+3})(1+3q^{2i+1}+3q^{4i+2}+q^{6i+3})}$$
(we multiplied the second terms in the numerator and denominator by $q^{6i+3}$).
Now since as $q\to 0$ the denominator is $1+o(1)$, it suffices to consider the leading term in the numerator, which is the leading term in
$$ \prod_{i=0}^\infty (1+3q^{2i+1}u+3q^{4i+2}v)(1+3q^{2i+1}v+3q^{4i+2}u),$$
where we dropped the $q^{6i+3}$ terms which are irrelevant.

To find the $u^jv^k$ term, we need to take the ``$u$" term from $j$ factors and the ``$v$'' term from $k$ factors.
Let $A=\prod_{i=0}^\infty (1+3q^{2i+1}u+3q^{4i+2}v)$ and $B=\prod_{i=0}^\infty(1+3q^{2i+1}v+3q^{4i+2}u).$
From $A$ we take $\ell_1$ of the $u$ terms, with leading-order coefficients $q^x$ for $x\in\Lc_1$ for some subset $\Lc_1\subset\{1,3,5,\dots\}$ of cardinality $\ell_1$.
Likewise from $B$ we take $\ell_2=j-\ell_1$ of the $u$ terms, with coefficients $q^{2x}$ for $x\in\Lc_2$ for some subset $\Lc_2\subset\{1,3,5,\dots\}$
of cardinality $\ell_2$. Likewise we take $m_1$ of the $v$ terms from $A$ and $m_2=k-m_1$ of the $v$ terms from $B$, corresponding to subsets $\Mc_1,\Mc_2\subset\{1,3,5,\dots\}$ of cardinalities $m_1,m_2$,
and the corresponding coefficients are to leading order $q^{2x}$ for $x\in \Mc_1$ and $q^x$ for $x\in \Mc_2$.
We require $\Lc_1\cap\Mc_1=\emptyset = \Lc_2\cap\Mc_2$. 

Let $L_1=\sum_{i\in\Lc_1}i$ and likewise define $L_2,M_1,M_2$. 
We need to make these choices to minimize the exponent of the leading-order term of $u^jv^k$ which is $L_1+2L_2+2M_1+M_2$.
Note that 
$$L_1+M_1 = 1+3+5+\dots+(2\ell_1+2m_1-1) = (\ell_1+m_1)^2$$
and
$$L_2+M_2 = 1+3+5+\dots+(2\ell_2+2m_2-1) = (\ell_2+m_2)^2= (j+k-\ell_1-m_1)^2.$$
So 
$$L_1+2L_2+2M_1+M_2=(\ell_1+m_1)^2+(j+k-\ell_1-m_1)^2 + L_2+M_1.$$
The first two terms here are independent of the choices of individual terms (just depending on $\ell_1,m_1$), so we can just choose the individual terms to minimize
$L_2$ and $M_1$ separately, that is, $\Lc_2 = \{1,3,\dots,(2j-2\ell_1-1)\}$ and $\Mc_1=\{1,3,\dots,2m_1-1\}$, giving $L_2=(j-\ell_1)^2$ and $\M_1=m_1^2$.

We are left with minimizing, for fixed $j,k$,
$$L_1+2L_2+2M_1+M_2=(\ell_1+m_1)^2+(j+k-\ell_1-m_1)^2 +(j-\ell_1)^2 + m_1^2,$$
subject to the constraints that $\ell_1,m_1,\ell_2=j-\ell_1,m_2=k-m_1\ge0$. 

By a short calculus exercise, we get the minimum exponent of 
$\frac23(j^2+jk+k^2)$.
This is the minimum for real $\ell_1$, $m_1$; taking into account the fact that the minimum must be an integer leads
to the exponent $\delta_{j,k}:=\lceil\frac23(j^2+jk+k^2)\rceil$.



\bibliographystyle{alpha}
\bibliography{references.bib}

\begin{thebibliography}{FLL19}

\bibitem[CD]{DanandTom}
T.~Cremaschi and D.~C. Douglas.
\newblock {Web basis for the $\mathrm{SL}(n)$-skein algebra of the annulus}.
\newblock In preparation.

\bibitem[FLL19]{Fraser19TransAmerMathSoc}
C.~Fraser, T.~Lam, and I.~Le.
\newblock {From dimers to webs}.
\newblock {\em Trans. Amer. Math. Soc.}, 371:6087--6124, 2019.

\bibitem[FP16]{FominAdvMath16}
S.~Fomin and P.~Pylyavskyy.
\newblock {Tensor diagrams and cluster algebras}.
\newblock {\em Adv. Math.}, 300:717--787, 2016.

\bibitem[FR99]{FockAmerMathSocTransl}
V.~V. Fock and A.~A. Rosly.
\newblock {Poisson structure on moduli of flat connections on Riemann surfaces
  and the $r$-matrix}.
\newblock In {\em Moscow Seminar in Mathematical Physics}, volume 191 of {\em
  Amer. Math. Soc. Transl. Ser. 2}, pages 67--86. Amer. Math. Soc., Providence,
  RI, 1999.

\bibitem[Fro19]{FrohmanJKnotTheoryRam19}
C.~Frohman.
\newblock {Spider evaluation and representations of web groups}.
\newblock {\em J. Knot Theory Ramifications}, 28:30 pp., 2019.

\bibitem[Jae92]{Jaeger92DiscreteMath}
F.~Jaeger.
\newblock {A new invariant of plane bipartite cubic graphs}.
\newblock {\em Discrete Math.}, 101:149--164, 1992.

\bibitem[Kas61]{Kasteleyn61Physica}
P.~W. Kasteleyn.
\newblock {The statistics of dimers on a lattice: I. The number of dimer
  arrangements on a quadratic lattice}.
\newblock {\em Physica}, 27:1209--1225, 1961.

\bibitem[Ken09]{Kenyon09Statisticalmechanics}
R.~Kenyon.
\newblock {Lectures on dimers}.
\newblock In {\em Statistical mechanics}, volume~16 of {\em IAS/Park City Math.
  Ser.}, pages 191--230. Amer. Math. Soc., Providence, RI, 2009.

\bibitem[Ken14]{Kenyon14CMP}
R.~Kenyon.
\newblock {Conformal invariance of loops in the double-dimer model}.
\newblock {\em Comm. Math. Phys.}, 326:477--497, 2014.

\bibitem[Kup96]{Kuperberg96CommMathPhys}
G.~Kuperberg.
\newblock {Spiders for rank 2 Lie algebras}.
\newblock {\em Comm. Math. Phys.}, 180:109--151, 1996.

\bibitem[LP09]{LP}
L\'{a}szl\'{o} Lov\'{a}sz and Michael~D. Plummer.
\newblock {\em Matching theory}.
\newblock AMS Chelsea Publishing, Providence, RI, 2009.
\newblock Corrected reprint of the 1986 original [MR0859549].

\bibitem[MP10]{Morse10Involve}
S.~Morse and E.~Peterson.
\newblock {Trace diagrams, signed graph colorings, and matrix minors}.
\newblock {\em Involve}, 3:33--66, 2010.

\bibitem[Sik01]{SikoraTrans01}
A.~S. Sikora.
\newblock {$\mathrm{SL}_n$-character varieties as spaces of graphs}.
\newblock {\em Trans. Amer. Math. Soc.}, 353:2773--2804, 2001.

\bibitem[SW07]{SikoraAlgGeomTop07}
A.~S. Sikora and B.~W. Westbury.
\newblock {Confluence theory for graphs}.
\newblock {\em Algebr. Geom. Topol.}, 7:439--478, 2007.

\bibitem[TF61]{Temperley61PhilosMag}
H.~N.~V. Temperley and M.~E. Fisher.
\newblock {Dimer problem in statistical mechanics--an exact result}.
\newblock {\em Philos. Mag.}, 6:1061--1063, 1961.

\bibitem[Thu90]{Thurston90AmerMathMonthly}
W.~P. Thurston.
\newblock {Conway's tiling groups}.
\newblock {\em Amer. Math. Monthly}, 97:757--773, 1990.

\end{thebibliography}

\end{document}